\documentclass[11pt]{amsart}

\usepackage[colorlinks, linkcolor=blue, citecolor=blue, urlcolor=red!80!gray, pagebackref]{hyperref}
\usepackage{xspace}
\usepackage{graphicx}
\usepackage{amssymb, amsmath, xcolor}
\usepackage{import}
\usepackage{tikz}
\usetikzlibrary{positioning,fit}
\usetikzlibrary{decorations.pathreplacing}
\usepackage{array}
\tikzset{filled/.style={minimum width=5pt,inner sep=0pt,circle,fill=black}}
\usepackage{subcaption}
\usepackage[indent]{parskip}
\usepackage{ytableau}
\usepackage{mathtools}
\usepackage{fullpage}
\usepackage[utf8]{inputenc}
\usepackage{float}
\usepackage{longtable}
\usepackage{todonotes}
\usepackage{amsrefs}
\usepackage{verbatim}
\usepackage{caption}
\usepackage[most]{tcolorbox}

\tikzstyle{dot}=[circle,fill, minimum size = 8pt, inner sep=0pt]

\newtheorem{theorem}{Theorem}[section]
\newtheorem{lemma}[theorem]{Lemma}
\newtheorem{corollary}[theorem]{Corollary}
\newtheorem{proposition}[theorem]{Proposition}
\newtheorem{conjecture}[theorem]{Conjecture}

\theoremstyle{definition}
\newtheorem{definition}[theorem]{Definition}

\newtheorem{construction}[theorem]{Construction}
\newtheorem{problem}[theorem]{Problem}

\theoremstyle{remark}

\numberwithin{equation}{section}
\numberwithin{figure}{section}

\newcommand{\N}{\mathbb{N}}
\newcommand{\SYT}{\operatorname{SYT}}
\newcommand{\sh}{\operatorname{sh}}
\newcommand{\Tdown}{T_{\lambda}}
\newcommand{\Tup}{T^{\uparrow}_\lambda}
\newcommand{\Blue}{cyan!20}
\newcommand{\Red}{red!20}
\DeclareDocumentCommand{\blue}{}{blue\xspace}
\DeclareDocumentCommand{\red}{}{red\xspace}

\title[RS shapes and cycle decompositions]{Robinson--Schensted shapes \\ arising from cycle decompositions}
\author[Erickson, Hunziker, Meddaugh, Sepanski, $\ldots$]{Martha Du Preez, William Q. Erickson, Jonathan Feigert, Markus Hunziker, Jonathan Meddaugh, Mitchell Minyard, Mark R. Sepanski,  Kyle Rosengartner}

\begin{document}


\begin{abstract}
    In the symmetric group $ S_n $, each element $ \sigma $ has an associated cycle type $ \alpha $, a partition of $ n $ that identifies the conjugacy class of $ \sigma $. 
    The Robinson–Schensted (RS) correspondence links each $ \sigma $ to another partition $ \lambda $ of $ n $, representing the shape of the pair of Young tableaux produced by applying the RS row-insertion algorithm to $ \sigma $. 
    Surprisingly, the relationship between these two partitions, namely the cycle type $ \alpha $ and the RS shape $ \lambda $, has only recently become a subject of study. 
    In this work, we explicitly describe the set of RS shapes $ \lambda $ that can arise from elements of each cycle type $ \alpha $ in cases where $ \alpha $ consists of two cycles.
    To do this, we introduce the notion of an \emph{$\alpha$-coloring}, where one colors the entries in a tableau of shape $\lambda$, in such a way as to construct a permutation $\sigma$ with cycle type $\alpha$ and RS shape $\lambda$.
         
\end{abstract}

\maketitle
\setcounter{tocdepth}{1}

\section{Introduction}

\subsection*{Statement of the problem}

At the heart of classical algebraic combinatorics is the representation theory of the symmetric group $S_n$.
In turn, much of this theory can be expressed in terms of integer partitions.
In this paper, we describe the subtle relationship between two partitions closely associated with each element $\sigma \in S_n$: the \emph{cycle type} of $\sigma$, on one hand, and the \emph{shape} of $\sigma$ (via the Robinson--Schensted correspondence) on the other hand.
Although separately each of these partitions is fundamental to the general theory, the two had not yet been studied \emph{together} until a very recent paper~\cite{RS-complete-Rubinstein} treating the special case where $\sigma$ is a cyclic (or almost cyclic) permutation.
(The most closely related work seems to be the paper~\cite{GesselReutenauer} on counting permutations with a given cycle type and number of descents.)
A natural question is the following: which shapes arise from the elements of a given cycle type?

It is well known that the conjugacy classes of $S_n$ (and also its irreducible complex representations) can be naturally labeled by the integer partitions $\alpha$ of $n$ (written as $\alpha \vdash n$).
In particular, the conjugacy class of $\sigma \in S_n$ is labeled by the partition $\alpha = (\alpha_1, \ldots, \alpha_r)$ giving the \emph{cycle type} of $\sigma$, which is easily read off from the expression of $\sigma$ in disjoint cycle notation:
\[
\sigma = (\text{$\alpha_1$-cycle})(\text{$\alpha_2$-cycle}) \cdots (\text{$\alpha_r$-cycle}).
\]
(As usual, we write partitions so that $\alpha_1 \geq \cdots \geq \alpha_r \geq 1$.)
Throughout the paper, we write $\mathcal{C}_{\alpha}$ to denote the conjugacy class of $S_n$ consisting of elements with cycle type $\alpha$.

Another key concept in the representation theory of $S_n$ (and in algebraic combinatorics in general) is the Robinson--Schensted (RS) correspondence~\cites{Robinson-Representations-of-Sn,SchenstedOriginalPaper}; see also the exposition in~\cites{Knuth-Perm-Mat-GYT, Krattenthaler-Inc/dec-chains, Fulton-YT, Sagan-Sn}.
The RS correspondence is a bijection
\[
S_n\ \xrightarrow{{\rm RS}} \ \coprod_{\lambda \vdash n} \, \SYT(\lambda) \times \SYT(\lambda),
\]
where $\SYT(\lambda)$ denotes the set of standard Young tableaux with shape $\lambda$, meaning that the partition $\lambda$ gives the row lengths of the tableaux.
(Note that this bijection can be viewed as the combinatorial analogue of the decomposition of the regular representation of $S_n$ into irreducible $S_n \times S_n$-modules.)
If the RS correspondence takes $\sigma$ to a pair $(P,Q) \in \SYT(\lambda) \times \SYT(\lambda)$, then we say that $\lambda$ is the \emph{RS shape} of $\sigma$, which we denote by writing $\sh(\sigma) = \lambda$.
Thus in the example
\begin{equation}
    \label{example intro}
    \ytableausetup{smalltableaux,centertableaux}
    \sigma = (3,5,4,7)(1,2,6) \xmapsto{{\rm RS}} \left( \: \ytableaushort{137,24,5,6}, \ytableaushort{124,37,5,6} \: \right),
\end{equation}
we have $\sigma \in \mathcal{C}_{(4,3)}$, and $\sh(\sigma) = (3,2,1,1)$.  
(See Figure~\ref{fig: ex of RSK} for the details in this example.)
The main problem in this paper is to describe the elements of 
\[
\mathcal{S}_\alpha \coloneqq \{ \sh(\sigma) : \sigma \in \mathcal{C}_\alpha \}.
\]

\subsection*{Main result}

As a preliminary result (see Theorem~\ref{thm:box}), for all $\alpha = (\alpha_1, \ldots, \alpha_r) \vdash n$, we prove that the partitions in $\mathcal{S}_\alpha$ have Young diagrams fitting inside a certain bounding box:
\begin{equation}
\label{sh fit in box}
\mathcal{S}_\alpha \subseteq  \mathcal{B}_{\alpha},
\end{equation}
where
\begin{equation}
\label{B}
   \mathcal{B}_\alpha 
   \coloneqq \left\{ \lambda \vdash n : \;\;\; \parbox{50ex}{$\lambda$ has at most \\ $\big( n - r + \#\{i : \alpha_i = 2\} + \delta_{1,\alpha_r} \big)$ many rows and \\
   $\big( n - r + \#\{i : \alpha_i = 1 \} \big)$ many columns} \right\}.
\end{equation}
For example, if $\alpha = (4,2)$, then $\mathcal{B}_\alpha$ consists of all partitions of $n=6$ whose Young diagram fits inside the box of dimensions $(6-2+1+0) \times (6-2+0) = 5 \times 4$.
Concretely, we have
\[
\mathcal{B}_{(4,2)} = 
\left\{
\:
\begin{tikzpicture}[scale=.3, baseline=(current bounding box.center)]
    \draw [dashed]
    (1,-1) rectangle (5,-6);

    \draw (1,-1) rectangle ++(1,-1);
    \draw (2,-1) rectangle ++(1,-1);
    \draw (3,-1) rectangle ++(1,-1);
    \draw (4,-1) rectangle ++(1,-1);
    \draw (1,-2) rectangle ++(1,-1);
    \draw (2,-2) rectangle ++(1,-1);

    \node at (5.5,-6) {,};
\end{tikzpicture}
\begin{tikzpicture}[scale=.3, baseline=(current bounding box.center)]
    \draw [dashed]
    (1,-1) rectangle (5,-6);

    \draw (1,-1) rectangle ++(1,-1);
    \draw (2,-1) rectangle ++(1,-1);
    \draw (3,-1) rectangle ++(1,-1);
    \draw (4,-1) rectangle ++(1,-1);
    \draw (1,-2) rectangle ++(1,-1);
    \draw (1,-3) rectangle ++(1,-1);

    \node at (5.5,-6) {,};
\end{tikzpicture}
\begin{tikzpicture}[scale=.3, baseline=(current bounding box.center)]
    \draw [dashed]
    (1,-1) rectangle (5,-6);

    \draw (1,-1) rectangle ++(1,-1);
    \draw (2,-1) rectangle ++(1,-1);
    \draw (3,-1) rectangle ++(1,-1);
    \draw (1,-2) rectangle ++(1,-1);
    \draw (2,-2) rectangle ++(1,-1);
    \draw (3,-2) rectangle ++(1,-1);

    \node at (5.5,-6) {,};
\end{tikzpicture}
\begin{tikzpicture}[scale=.3, baseline=(current bounding box.center)]
    \draw [dashed]
    (1,-1) rectangle (5,-6);

    \draw (1,-1) rectangle ++(1,-1);
    \draw (2,-1) rectangle ++(1,-1);
    \draw (3,-1) rectangle ++(1,-1);
    \draw (1,-2) rectangle ++(1,-1);
    \draw (2,-2) rectangle ++(1,-1);
    \draw (1,-3) rectangle ++(1,-1);

    \node at (5.5,-6) {,};
\end{tikzpicture}
\begin{tikzpicture}[scale=.3, baseline=(current bounding box.center)]
    \draw [dashed]
    (1,-1) rectangle (5,-6);

    \draw (1,-1) rectangle ++(1,-1);
    \draw (2,-1) rectangle ++(1,-1);
    \draw (3,-1) rectangle ++(1,-1);
    \draw (1,-2) rectangle ++(1,-1);
    \draw (1,-3) rectangle ++(1,-1);
    \draw (1,-4) rectangle ++(1,-1);

    \node at (5.5,-6) {,};
\end{tikzpicture}
\begin{tikzpicture}[scale=.3, baseline=(current bounding box.center)]
    \draw [dashed]
    (1,-1) rectangle (5,-6);

    \draw (1,-1) rectangle ++(1,-1);
    \draw (2,-1) rectangle ++(1,-1);
    \draw (1,-2) rectangle ++(1,-1);
    \draw (2,-2) rectangle ++(1,-1);
    \draw (1,-3) rectangle ++(1,-1);
    \draw (2,-3) rectangle ++(1,-1);

    \node at (5.5,-6) {,};
\end{tikzpicture}
\begin{tikzpicture}[scale=.3, baseline=(current bounding box.center)]
    \draw [dashed]
    (1,-1) rectangle (5,-6);

    \draw (1,-1) rectangle ++(1,-1);
    \draw (2,-1) rectangle ++(1,-1);
    \draw (1,-2) rectangle ++(1,-1);
    \draw (2,-2) rectangle ++(1,-1);
    \draw (1,-3) rectangle ++(1,-1);
    \draw (1,-4) rectangle ++(1,-1);

    \node at (5.5,-6) {,};
\end{tikzpicture}
\begin{tikzpicture}[scale=.3, baseline=(current bounding box.center)]
    \draw [dashed]
    (1,-1) rectangle (5,-6);

    \draw (1,-1) rectangle ++(1,-1);
    \draw (2,-1) rectangle ++(1,-1);
    \draw (1,-2) rectangle ++(1,-1);
    \draw (1,-3) rectangle ++(1,-1);
    \draw (1,-4) rectangle ++(1,-1);
    \draw (1,-5) rectangle ++(1,-1);

    \node at (4,-6) {\phantom{,}};
\end{tikzpicture}
\:
\right\}.
\]

For certain cycle types $\alpha$, the containment in~\eqref{sh fit in box} is, in fact, an equality.
We can thus reframe our main problem as follows: classify the cycle types $\alpha$ such that $\mathcal{S}_\alpha = \mathcal{B}_\alpha$, and for the remaining cycle types $\alpha$, determine the complement $\mathcal{B}_{\alpha} \setminus \mathcal{S}_\alpha$.
The following theorem is the main result of this paper, where we solve the problem in the case $r=2$, that is, where $\alpha = (\alpha_1, \alpha_2)$.

\begin{theorem}
    \label{thm:main result}

    Let $n$ be a positive integer, and let $\alpha = (\alpha_1, \alpha_2) \vdash n$.

    \begin{enumerate}
        \item If $n$ is odd, then $\mathcal{S}_\alpha = \mathcal{B}_\alpha$.
        
        \item If $n$ is even, then $\mathcal{S}_\alpha = \mathcal{B}_\alpha$ unless $\alpha$ occurs in the following table:

        \end{enumerate}
    \[
    \renewcommand{\arraystretch}{1.5}
    \begin{array}{l|l}
        \alpha & \mathcal{B}_{\alpha} \setminus \mathcal{S}_\alpha \\ \hline
        (n-1, \: 1) & \{\left(\frac{n}{2}, \frac{n}{2}\right)\} \\
        \left(\frac{n}{2}, \frac{n}{2}\right), \textup{ where $4 \mid n$} & \{(n-2, 1, 1), \: (3,1, \ldots, 1)\} \\
        \left(\frac{n}{2}, \frac{n}{2}\right), \textup{ where $4 \nmid n$} & \{(n-2, 1, 1)\} \\
        (4,2) & \{(2,2,2)\} \\
        (5,3) & \{(2,2,2,2)\}
    \end{array}
    \]
    \end{theorem}

Our Theorem~\ref{thm:main result} generalizes the main result of~\cite{RS-complete-Rubinstein}, which can be restated as follows, using the language of the present paper:
\begin{enumerate}
    \item If $n$ is odd, then $\mathcal{S}_{(n)} = \mathcal{B}_{(n)}$ and $\mathcal{S}_{(n-1,1)} = \mathcal{B}_{(n-1,1)}$.
    \item If $n$ is even, then $\mathcal{S}_{(n)} = \mathcal{B}_{(n)}$ and $\mathcal{B}_{(n-1,1)} \setminus \mathcal{S}_{(n-1,1)} = \left\{\left(\frac{n}{2}, \frac{n}{2}\right)\right\}$.
    
\end{enumerate}
We emphasize that by introducing the bounding box $\mathcal{B}_\alpha$, we are able to state these results in a uniform manner, thus avoiding the need to designate certain ``trivial shapes'' and to make exceptions for small values of $n$. 
(Compare with the original formulation in~\cite{RS-complete-Rubinstein} given in Theorem~\ref{thm: original paper main result} below.)

\subsection*{Admissible tableaux and $\alpha$-colorings}

In order to prove Theorem~\ref{thm:main result} (see the end of Section~\ref{sec: unattainable}), we introduce two key combinatorial objects, which we call \emph{admissible tableaux} and \emph{$\alpha$-colorings}.
(See Definitions~\ref{def:admissible} and~\ref{def: alpha coloring}.)
We say a standard tableau is \emph{admissible} if it remains standard when justified along the bottom (i.e., when acted on by ``gravity'').
It turns out (see Proposition~\ref{prop:MM}) that the admissibility of a tableau $Q$ is equivalent to the following: for any tableau $P$ of the same shape, the permutation ${\rm RS}^{-1}(P,Q)$ can be read off from $P$ by following the order of the entries in the vertical reflection of $Q$ (which we denote by $Q^\uparrow$).
In turn, for an admissible $Q \in \SYT(\lambda)$, an \emph{$\alpha$-coloring} of $Q^\uparrow$ is a coloring of its entries which, via a certain canonical spiral construction, produces a permutation $\sigma \in \mathcal{C}_\alpha$ with $\sh(\sigma) = \lambda$.
This $\sigma$ cyclically permutes the entries of each color in $Q^\uparrow$ to obtain a standard $P$ such that $\sigma = {\rm RS}^{-1}(P,Q)$.
The crux of the paper is therefore Section~\ref{sec:alpha-colorings}, where we construct $\alpha$-colorings for all shapes $\lambda$, with the exception of the pairs $(\alpha,\lambda)$ given above in the table in Theorem~\ref{thm:main result}.
(There is also one family $(\alpha,\lambda)$ such that $\lambda \in \mathcal{S}_\alpha$ but no $\alpha$-coloring exists; in this case we exhibit the requisite $\sigma$ directly.)

\subsection*{Open problems and conjectures}

As a first step toward extending the results in this paper to generic cycle types $(\alpha_1, \ldots, \alpha_r)$ where $r > 2$, we point out a special case in which we can describe $\mathcal{S}_\alpha$ explicitly, namely when $\alpha_1 \leq 2$:
\[
\mathcal{S}_{(2^{r-k}, \: 1^{k})} = \left\{ \lambda \vdash n : \text{$\lambda$ has exactly $k$ many columns of odd length} \right\}.
\]
(See Proposition~\ref{prop:schutz}, following from a result of Sch\"utzenberger relating to involutions.)
Outside of this special case described above, however, for $r>2$, an explicit and comprehensive description of $\mathcal{S}_\alpha$ quickly becomes quite complicated.
Somewhat surprisingly, the source of these complications lies entirely in the presence of repeated values among the $\alpha_i$'s.
In fact (see Conjecture~\ref{conj: distinct ct}), it appears that for $r > 2$, we have $\mathcal{S}_{\alpha} = \mathcal{B}_{\alpha}$ whenever $\alpha_1 > \cdots > \alpha_r > 1$.
(It would then follow that for a fixed $r$, as $n \rightarrow \infty$, the proportion of cycle types satisfying $\mathcal{S}_{\alpha} = \mathcal{B}_{\alpha}$ approaches 100\%.)
See Section~\ref{sec:conjectures} for additional details, along with further conjectures and open problems concerning cycle types and $\alpha$-colorings.

\section{Preliminaries}

\subsection*{Partitions}

We write $\N$ for the set of positive integers.
For $n \in \N$, a \emph{partition} of $n$ is a weakly decreasing finite sequence $\lambda = (\lambda_1, \ldots, \lambda_\ell)$ of positive integers, such that $\sum_{i=1}^\ell \lambda_i = n$.
We call $n$ the \emph{size} of $\lambda$, which we express by writing $\lambda \vdash n$.
We will often use exponents to denote repeated parts in a partition; for example, $(4^3,2,1^2)$ is shorthand for $(4,4,4,2,1,1)$.
The \emph{Young diagram of shape $\lambda$} is the left-justified array with $\lambda_i$ boxes in the $i$th row, counting from the top. 
As an example, the Young diagram of shape $\lambda = (3,2,1,1)$ is given in Figure~\ref{subfig:YD}.
For brevity, we often refer to the ``rows/columns of $\lambda$'' or similar phrases, where it is understood that we are speaking of the Young diagram of $\lambda$.

Given the Young diagram of a partition $\lambda$, the \emph{tail} of the diagram is the part of the first row that extends beyond the second row; in other words, the tail consists of the rightmost $\lambda_1 - \lambda_2$ many boxes in the first row.
When $\lambda_1 > \lambda_2 = 1$, we say that $\lambda$ is a \emph{hook}, due to the shape of its Young diagram (the union of a single row and a single column).

Since we will often consider the column lengths rather than the row lengths, we recall the \emph{conjugate partition} of $\lambda$, denoted by $\lambda' = (\lambda'_1, \ldots, \lambda'_m)$, where  $m = \lambda_1$, and where $\lambda'_j$ denotes the length of the $j$th column of the Young diagram of $\lambda$, or equivalently
\[
\lambda'_j=\#\{i : \lambda_i \geq j\}.
\]
For example, if $\lambda = (3,2,1,1)$ is the partition with Young diagram shown in Figure~\ref{subfig:YD}, then we have $\lambda' = (4,2,1)$.
We emphasize that whenever we depict $\lambda$ by its Young diagram, the row lengths are given by $\lambda_1, \ldots, \lambda_\ell$, while the column lengths are given by $\lambda'_1, \ldots, \lambda'_m$.

\subsection*{The symmetric group and related notation}

Let $S_n$ denote the symmetric group on $n$ letters.
One can view elements of $S_n$ as permutations $\sigma = [\sigma_1, \ldots, \sigma_n]$ of the list $(1,\ldots, n)$.
This is known as the \emph{one-line notation} for $\sigma$.
In this paper, we will primarily express elements $\sigma \in S_n$ using \emph{cycle notation}, namely, as a product of disjoint cycles written in rounded parentheses.
As an example of these two methods of notation, we have
\[
[2, 4, 9, 7, 10, 3, 1, 14, 6, 12, 11, 5, 13, 8] = (1,2,4,7) (3, 9, 6) (5,10,12) (8, 14) (11)(13).
\]
The \emph{cycle type} of an element $\sigma \in S_n$ is the partition given by the lengths of the disjoint cycles in the cycle notation for $\sigma$.
In the example displayed above, $\sigma$ has cycle type $(4,3,3,2,1,1)$.
It is well known that the conjugacy classes in $S_n$ are precisely the equivalence classes induced by cycle type.
In other words, there is a bijection between partitions of $n$ and conjugacy classes of $S_n$, where each partition $\alpha \vdash n$ corresponds to the conjugacy class 
\[
\mathcal{C}_\alpha \coloneqq \{\sigma \in S_n : \sigma \text{ has cycle type } \alpha \}.
\]

\begin{figure}[t]
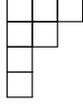
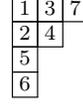

     \centering
     \ytableausetup{smalltableaux}
     \begin{subfigure}[b]{0.4\textwidth}
         \centering
         \ydiagram{3,2,1,1}
         
         \caption{The Young diagram of shape $\lambda$.}
         \label{subfig:YD}
     \end{subfigure}
     \qquad
     \begin{subfigure}[b]{0.4\textwidth}
         \centering
         \ytableaushort{137,24,5,6}
         \caption{An element of $\SYT(\lambda)$.}
         \label{subfig:SYT}
     \end{subfigure}
     \caption{Example where $\lambda = (3,2,1,1)$.
     Note that $\lambda' = (4,2,1)$.}
        \label{fig:YD and SYT}
\end{figure}

\subsection*{The Robinson--Schensted correspondence}

For $\lambda \vdash n$, a \emph{standard Young tableau of shape $\lambda$} is a Young diagram of shape $\lambda$, where the $n$ boxes are filled with distinct entries from $\{1, \ldots, n\}$ such that every row and column is increasing.
(See Figure~\ref{subfig:SYT} for an example.)
Let $\SYT(\lambda)$ denote the set of all standard Young tableaux of shape $\lambda$.
As described in the introduction, the \emph{Robinson--Schensted (RS) correspondence} is a bijection 
\begin{equation}
\label{RSK bijection}
S_n\ \xrightarrow{{\rm RS}}\ \coprod_{\lambda \vdash n} \SYT(\lambda) \times \SYT(\lambda),
\end{equation}
which we give explicitly below.
(See also the geometric \emph{light-and-shadow} interpretation due to Viennot~\cite{Viennot}, \emph{growth diagrams} due to Fomin \cite{StanleyEC2}*{\S 7.13},
and the \emph{jeu de taquin} interpretation \cite{StanleyEC2}*{\S A1.3}.)

Let $\sigma = [\sigma_1, \ldots, \sigma_n] \in S_n$, and let $(P,Q)$ be the image of $\sigma$ under the RS correspondence.
We write this as $\operatorname{RS}(\sigma) = (P,Q)$.
The tableaux $P$ and $Q$ are constructed recursively as follows.
Begin by defining both $P_0 = P_0(\sigma)$ and $Q_0 = Q_0(\sigma)$ to be the empty tableau.
Then for each $i = 1, \ldots, n$, we construct the tableaux $P_i = P_i(\sigma)$ and $Q_i = Q_i(\sigma)$ as follows.
The tableau $P_i$ is obtained from $P_{i-1}$ via the following \emph{row insertion} algorithm:
\begin{enumerate}
    \item Set $a_1 \coloneqq \sigma_i$, and initialize $k=1$.
    \item If $a_k$ is greater than the rightmost entry in row $k$ of $P_{i-1}$, then add a box labeled $a_k$ to the end of that row, and terminate the construction of $P_i$. 
    \item Otherwise, set $a_{k+1}$ equal to the leftmost entry in row $k$ of $P_{i-1}$ which is greater than $a_k$, and replace this entry with $a_k$.
    Now increment $k$ by $1$.
    \item Repeat steps (2) and (3) until the process terminates, resulting in the tableau $P_i$.
\end{enumerate}
(Step 3 above is often described as the entry $a_k$ ``bumping'' the entry $a_{k+1}$ into the next row.)
Note that the shape of $P_i$ is that of $P_{i-1}$ with the addition of one new box; to construct $Q_i$, add a new box labeled $i$ to the corresponding position in $Q_{i-1}$.
After repeating this process to construct $(P_1,Q_1), \ldots, (P_n,Q_n)$, the pair $(P,Q)$ is taken to be the final pair $(P_n, Q_n)$.
(See Figure \ref{fig: ex of RSK} for an example of this algorithm, the result of which was shown in~\eqref{example intro} in the introduction.)
Note that $P$ and $Q$ share the same shape; if this shape is $\lambda$, then we call $\lambda$ the \emph{RS shape} of $\sigma$, and we write $\sh(\sigma) = \lambda$.

\begin{figure}[t]
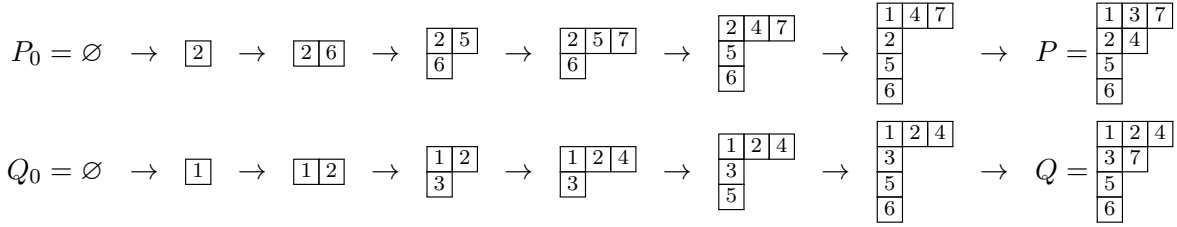
 
    \centering

    \[
    \begin{array}{rcccccccccccccl}
    \ytableausetup{smalltableaux}
    P_0 = \varnothing & \rightarrow & \ytableaushort{2}
    & \rightarrow &
    \ytableaushort{26}
    &\rightarrow&
    \ytableaushort{25,6}
    &\rightarrow&
    \ytableaushort{257,6}
    &\rightarrow&
    \ytableaushort{247,5,6}
    &\rightarrow&
    \ytableaushort{147,2,5,6}
    &\rightarrow&
    P = 
    \ytableaushort{137,24,5,6} \\[4ex]
    Q_0 = \varnothing & \rightarrow & \ytableaushort{1}
    & \rightarrow &
    \ytableaushort{12}
    &\rightarrow&
    \ytableaushort{12,3}
    &\rightarrow&
    \ytableaushort{124,3}
    &\rightarrow&
    \ytableaushort{124,3,5}
    &\rightarrow&
    \ytableaushort{124,3,5,6}
    &\rightarrow&
    Q = 
    \ytableaushort{124,37,5,6}

    \end{array}
    \]
    \caption{Details of the example~\eqref{example intro} in the introduction.
    Given the input $\sigma = [2,6,5,7,4,1,3] = (3,5,4,7)(1,2,6)$, the RS algorithm produces the output $\operatorname{RS}(\sigma) = (P,Q)$.
    Note that $\sh(\sigma) = (3,2,1,1)$.}
    \label{fig: ex of RSK}
\end{figure}

If $\sh(\sigma) = \lambda$, then the row lengths $\lambda_i$ and column lengths $\lambda'_j$ encode information regarding the increasing and decreasing subsequences in the one-line notation of $\sigma$, as follows.
Let 
\begin{align*}
    a_k(\sigma)\coloneqq & \; \text{maximum length of the union of} \\[-1ex]
    &\; \text{$k$ disjoint \emph{ascending} subsequences in $[\sigma_1, \ldots, \sigma_n]$},\\
    d_k(\sigma)\coloneqq & \; \text{maximum length of the union of} \\[-1ex]
    &\; \text{$k$ disjoint \emph{descending} subsequences in $[\sigma_1, \ldots, \sigma_n]$}.
\end{align*}
Generalizing a theorem of Schensted, Greene \cite{Greene--Schensted-Extension}*{Thm.~3.1} showed  that
\begin{equation} \label{eq: desc asc description of cols and rows of RSK shape}
    a_k(\sigma) = \sum_{i=1}^k \lambda_i \quad \text{and} \quad d_k(\sigma) = \sum_{j=1}^k \lambda'_j.
\end{equation}
In the special case $k=1$ (Schensted's original theorem), this says that the number of columns (resp., rows) in $\lambda$ equals the length of the longest ascending (resp., descending) subsequence in $[\sigma_1, \ldots, \sigma_n]$.

\subsection*{Main problem and preliminary results}

Having defined the conjugacy classes $\mathcal{C}_\alpha$ in $S_n$, along with the RS shape of an element $\sigma \in S_n$, we now introduce the main object of study in this paper, namely the set of RS shapes that arise from the elements of a given conjugacy class.
Given a cycle type $\alpha \vdash n$, define
\begin{equation}
    \label{S_alpha}
    \mathcal{S}_\alpha \coloneqq \{ \sh(\sigma) : \sigma \in \mathcal{C}_\alpha\}.
\end{equation}
The recent paper~\cite{RS-complete-Rubinstein} described the set $\mathcal{S}_\alpha$ in the three general cases below.

\begin{theorem}[{\cite{RS-complete-Rubinstein}*{Theorems 4.2, 5.1}}] \label{thm: original paper main result}
    Let $n\in\N$. Then:
    \begin{enumerate}
        \item $\mathcal{S}_{(1)} = \{(1)\}$, $\mathcal{S}_{(2)} = \{(1^2)\}$, and, for $n\geq 3$, we have
        $\mathcal{S}_{(n)} = \{ \lambda \vdash n \} \setminus \{(1^n), (n)\}$.
        \item $\mathcal{S}_{(2,1)} = \{(2,1), \: (1^3) \}$ and, for odd $n\geq  5$, we have $\mathcal{S}_{(n-1, 1)} = \{ \lambda \vdash n \} \setminus \{(1^{n}), \: (n)\}$.
        \item $\mathcal{S}_{(1,1)} = \{(2)\}$ and, for even $n\geq 4$, we have $\mathcal{S}_{(n-1, 1)} = \{\lambda \vdash n\} \setminus \{(1^{n}), \: (n), \: \left(\frac{n}{2}, \frac{n}{2}\right)\}$.
    \end{enumerate}
\end{theorem}

Our main result (Theorem~\ref{thm:main result}, stated in the introduction) generalizes Theorem~\ref{thm: original paper main result} by describing $\mathcal{S}_\alpha$ for all cycle types $\alpha = (\alpha_1, \alpha_2)$.
To this end, we begin by establishing a general bound on the shapes which can appear in $\mathcal{S}_\alpha$.
Given $\alpha = (\alpha_1, \ldots, \alpha_r) \vdash n$, we recall from the introduction~\eqref{B} the following set of partitions whose Young diagram fits inside a certain box:

\begin{equation}
\label{B again}
   \mathcal{B}_\alpha 
   \coloneqq \left\{ \lambda 
   \vdash n : \;\; \parbox{35ex}{$\lambda_1' \leq n - r + \#\{i : \alpha_i = 2\} + \delta_{1,\alpha_r}$, \\
   $\lambda_1 \leq n - r + \#\{i : \alpha_i = 1 \}$} \right\},
\end{equation}
where $\delta_{1, \alpha_r}$ is the Kronecker delta (taking the value 1 if $\alpha$ contains a 1, and 0 otherwise).

\begin{theorem}
\label{thm:box}
    Let $\alpha \vdash n$, with $\mathcal{S}_\alpha$ and $\mathcal{B}_\alpha$ as in~\eqref{S_alpha} and~\eqref{B again}.
    We have $  \mathcal{S}_\alpha \subseteq \mathcal{B}_\alpha$.
\end{theorem}

\begin{proof}
   Let $\lambda = (\lambda_1, \ldots, \lambda_\ell) \in \mathcal{S}_\alpha$.
    By definition~\eqref{S_alpha}, there exists some $\sigma \in \mathcal{C}_\alpha$ such that $\lambda = \sh(\sigma)$.
    By Schensted's theorem (i.e., the $k=1$ case of Greene's theorem~\eqref{eq: desc asc description of cols and rows of RSK shape}), the length $\ell = \lambda'_1$ equals the length of the longest descending subsequence in $[\sigma_1, \ldots, \sigma_n]$.
    Meanwhile, $\sigma$ consists of $r$ disjoint cycles of lengths $\alpha_1, \ldots, \alpha_r$.
    Let $d_i$ denote the length of the longest descending subsequence in the cycle of length $\alpha_i$ (where the elements in the cycle are written so as to preserve the order in the one-line notation of $\sigma$).
    Then clearly the length $\ell$ of the longest descending subsequence in $[\sigma_1, \ldots, \sigma_n]$ is at most the sum $d_1 + \cdots + d_r$.
    
    Now, again by Schensted's theorem, each $d_i$ is the number of rows in the RS shape of some element of $\mathcal{C}_{(\alpha_i)} \subset S_{\alpha_i}$.
    Thus by inspecting the maximum number of rows in the shapes in $\mathcal{S}_{(\alpha_i)}$ given in part (1) of Theorem~\ref{thm: original paper main result}, we have
    \[
    d_i \leq \begin{cases}
        1, & \alpha_i = 1,\\
        2, & \alpha_i = 2,\\
        \alpha_i-1, & \alpha_i \geq 3.
    \end{cases}
    \]
    Hence we have
    \begin{equation}
        \label{d inequality}
     \ell \leq d_1 + \cdots + d_r \leq \underbrace{\sum_{i=1}^r (\alpha_i - 1)}_{n-r} + \, \#\{i: \alpha_i = 2\} + \#\{i : \alpha_i = 1\}.
    \end{equation}
    But in fact, since 1-cycles are fixed points $\sigma_j = j$, repeated 1's among the $\alpha_i$'s form a single increasing subsequence in $[\sigma_1, \ldots, \sigma_n]$; hence, repeated 1's cannot possibly contribute more than a total of 1 to the length of a descending subsequence.
    Therefore the bound~\eqref{d inequality} on $\ell$ can be sharpened, by replacing $\#\{i: \alpha_i = 1\}$ by either 1 (if $\alpha$ contains a 1) or 0 (otherwise).
    This yields the bound on $\ell$ given in the definition of $\mathcal{B}_\alpha$ in~\eqref{B again}.

    The argument bounding $\lambda_1$ is similar, upon replacing ``descending'' by ``ascending'' in the proof above.
    In this case, letting $a_i$ denote the length of the longest ascending subsequence in the cycle of length $\alpha_i$ in $\sigma$, we immediately have that $\lambda_1 \leq a_1 + \cdots + a_r$.
    By inspecting the maximum number of columns in the shapes in $\mathcal{S}_{(\alpha_i)}$ given in part (1) of Theorem~\ref{thm: original paper main result}, we have
    \[
    a_i \leq \begin{cases}
        1, & \alpha_i = 1,\\
        \alpha_i-1, & \alpha_i \geq 2.
    \end{cases}
    \]
    This time, repeated 1's can indeed contribute 1 individually (not just 1 in total) to an ascending subsequence in $[\sigma_1, \ldots, \sigma_n]$, and so we are left with the inequality
    \[
     \lambda_1 \leq a_1 + \cdots + a_r \leq \underbrace{\sum_{i=1}^r (\alpha_i - 1)}_{n-r} + \#\{i: \alpha_i = 1\}.
    \]
    This yields the bound on $\lambda_1$ given in~\eqref{B again}, and so completes the proof.
\end{proof}

\textit{A priori} it is not clear whether the containment $\mathcal{S}_\alpha \subseteq \mathcal{B}_\alpha$ is sharp (i.e., whether there exist any $\alpha$'s such that $\mathcal{S}_\alpha = \mathcal{B}_\alpha$).
It turns out that the answer is affirmative (as summarized in the introduction).
The remainder of the paper is devoted to proving our main result (see Theorem~\ref{thm:main result} above), which classifies the cycle types $\alpha = (\alpha_1, \alpha_2)$ such that $\mathcal{S}_{\alpha} = \mathcal{B}_{\alpha}$, and determines the complement $\mathcal{B}_\alpha \setminus \mathcal{S}_\alpha$ for the remaining $\alpha$'s.

\section{Admissible tableaux and canonical cycles}
\label{sec: fund cyc const}

In this section (see Construction~\ref{constr:canonical cycle}), for each shape $\lambda \in \mathcal{B}_{(n)}$, we construct a canonical cycle $\sigma \in \mathcal{C}_{(n)}$ such that $\sh(\sigma) = \lambda$.
As it will be seen when we define $\alpha$-colorings in Section~\ref{sec:alpha-colorings}, this construction of canonical cycles can be viewed as an $\alpha$-coloring in the special case where $r=1$.
En route to constructing $\alpha$-colorings, we also introduce an important class of standard tableaux which we call \emph{admissible} (see Definition~\ref{def:admissible}).

The group $S_n$ acts naturally on a tableau $T$ of size $n$ by permuting entries: in particular, $\sigma \cdot T$ is the tableau obtained from $T$ by replacing each entry $i$ with $\sigma_i$.
(In this setting, neither $T$ nor $\sigma \cdot T$ needs to be standard.)
We define the \emph{column word of} a tableau to be the string of entries obtained by reading the columns from top to bottom, and concatenating from left to right.
By contrast, the \emph{reverse column word} is obtained by reading the entries in each column from bottom to top instead.
We write column words in square brackets to align with our one-line notation for permutations.
A distinguished role will be played by the tableau
\begin{equation}
\label{Tdown}
    \Tdown \coloneqq \text{the tableau in $\SYT(\lambda)$ with column word $[1,\ldots, n]$}.
\end{equation}
Given a standard Young tableau $Q$, we define
\[
Q^\uparrow \coloneqq \text{the tableau obtained from $Q$ by reversing the entries in each column.}
\]
Equivalently, the column word of $Q^\uparrow$ is the reverse column word (in the above sense) of $Q$.
Note that $Q^\uparrow$ is generally not a standard tableau: indeed, the entries \emph{decrease} from top to bottom in each column, and may or may not still increase from left to right in each row.

\begin{definition}[Admissible tableaux]
    \label{def:admissible}
    A standard tableau $Q$ is said to be \emph{admissible} if the entries in $Q^\uparrow$ increase from left to right along each row.
\end{definition}

An equivalent formulation of Definition~\ref{def:admissible} can be given as follows.
By applying ``gravity'' to the boxes in $Q$, one obtains a bottom-justified tableau with the same column lengths as $Q$.
We say that $Q$ is admissible if this bottom-justified tableau is standard in the usual sense (i.e., entries increase down columns and along rows).
It seems that the proportion of admissible tableaux becomes vanishingly small as $n \rightarrow \infty$; already when $n=12$, for instance, roughly $2\%$ of standard tableaux are admissible.
(See Problem~\ref{prob:admissible} below.)
The following proposition explains the significance of admissible tableaux.


\begin{proposition}
    \label{prop:MM}
    Let $\lambda \vdash n$.
    A tableau $Q \in \SYT(\lambda)$ is admissible if and only if for all $P \in \SYT(\lambda)$, we have
    \[
    P = \sigma \cdot Q^\uparrow,
    \]
    where $\sigma \coloneqq \operatorname{RS}^{-1}(P,Q)$.
\end{proposition}

\begin{proof}
Let $P,Q \in \SYT(\lambda)$, and suppose that $Q$ is admissible.
We must show that $P = \sigma \cdot Q^\uparrow$.
To do this, let $\rho \in S_n$ be the unique permutation such that $P = \rho \cdot Q^\uparrow$ (i.e., $\rho_i$ is the entry in $P$ with the same position as the entry $i$ in $Q^\uparrow$).
It suffices to show that $\rho = \sigma$, or in other words (since the RS correspondence is a bijection), that $\operatorname{RS}(\rho) = (P,Q)$.

We prove this by induction on the index $i = 1, \ldots, n$, where as usual $P_i(\rho)$ and $Q_i(\rho)$ denote the tableaux obtained after the $i$th step of the RS algorithm applied to $\rho$.
As our induction hypothesis, we assume that
\begin{enumerate}
\item Every entry in $P_i(\rho)$ lies in the column corresponding to its column in $P$.
\item The tableau $Q_i(\rho)$ is the subtableau of $Q$ consisting of entries $1, \ldots, i$.
\end{enumerate}
In the base case $i=1$, the entry $1$ occupies the lower-left corner of $Q^\uparrow$ and therefore the entry $\rho_1$ is also in the lower-left corner (thus in column 1) of $P$.
Therefore since $P_1(\rho)$ is the single box with entry $\rho_1$, and $Q_1(\rho)$ is the single box with entry $1$, we have verified both hypotheses (1) and (2) above in the base case.

Now to show that hypotheses (1) and (2) remain true upon replacing $i$ by $i+1$, we consider the effect of inserting $\rho_{i+1}$ into the tableau $P_{i}(\rho)$ to obtain $P_{i+1}(\rho)$.
Because $Q$ is both standard and admissible, the entries in $Q^\uparrow$ increase upwards along columns and left-to-right along rows; therefore every entry lying weakly southwest of $i+1$ in $Q^\uparrow$ is less than $i+1$.
It follows that every entry lying weakly southwest of $\rho_{i+1}$ in $P$ has already been inserted (i.e., is one of the previous $\rho_1, \ldots, \rho_{i}$).
By the same reasoning, every entry lying weakly northeast of $\rho_{i+1}$ in $P$ has \emph{not} yet been inserted.
Thus if $k$ denotes the index of the column in $P$ containing the entry $\rho_{i+1}$, then the entries in column $k$ that have already been inserted are precisely those lying below $\rho_{i+1}$; moreover, by the induction hypothesis (1), these entries are precisely the entries in the column $k$ of the tableau $P_{i}(\rho)$.
In particular, the topmost entry of column $k$ in $P_i(\rho)$ is the entry directly below (and therefore greater than) $\rho_{i+1}$ in $P$.
Therefore $\rho_{i+1}$ must be inserted either into column $k$ of $P_{i}(\rho)$, or into some column to the left.

By hypothesis (1) and the admissibility of $Q$, column $k-1$ of $P_i(\rho)$ consists of all the entries in column $k-1$ of $P$ which lie weakly southwest of $\rho_{i+1}$ (along with possibly some entries directly above them in column $k-1$ of $P$).
Thus the topmost entry in column $k-1$ of $P_i(\rho)$ is some entry in column $k-1$ of $P$ which is weakly above (and therefore less than) $\rho_{i+1}$.
It follows that $\rho_{i+1}$ must in fact be inserted into column $k$ of $P_i(\rho)$.
It also follows that column $k-1$ is strictly longer than column $k$ in $P_i(\rho)$, and so by the same argument above, this insertion of $\rho_{i+1}$ bumps each entry in column $k$ of $P_i(\rho)$ directly downward (since it is greater than the entry below it in column $k-1$).
Thus hypothesis (1) is true of $P_{i+1}(\rho)$.
Furthermore, since the new box in $P_{i+1}(\rho)$ is added to column $k$, it follows that $Q_{i+1}(\rho)$ is obtained from $Q_i(\rho)$ by adding a box with entry $i+1$ to the bottom of column $k$.
By hypothesis (2) we have that $Q_{i}(\rho)$ is already the subtableau of $Q$ consisting of entries $1, \ldots, i$.
Therefore since column $k$ is the column of $Q$ containing the entry $i+1$, we conclude that hypothesis (2) is true of $Q_{i+1}(\rho)$.
Having completed the induction argument, we now have that $P_n(\rho) = P$, since a standard tableau is determined by its columns, and $Q_n(\rho) = Q$.
Therefore we have $\operatorname{RS}(\rho) = (P,Q)$, as desired, and thus $\rho = \sigma$.

To prove the converse, suppose that $Q \in \SYT(\lambda)$ is \emph{not} admissible.
We will produce a tableau $P$ such that $P \neq \sigma \cdot Q^\uparrow$.
Set $j$ equal to the smallest entry in $Q^\uparrow$ which is less than the entry directly to its left.
(By definition of admissibility, such an entry must exist.)
Let $k$ denote the index of the column containing this entry $j$.
Since $Q$ is standard, column $k-1$ of $Q^\uparrow$ must contain at least one entry which is less than $j$; let $i$ be the smallest such entry.
Then we have $i < j$ and $i$ lies strictly southwest of $j$.
Now choose $P \in \SYT(\lambda)$ so that $\rho_{i} > \rho_j$ (where as before, $\rho$ is the unique permutation such that $P = \rho \cdot Q^\uparrow$); this is always possible since $\rho_i$ lies strictly southwest of $\rho_j$ in $P$.
By the minimality of $j$, the first part of the proof guarantees that $P_{j-1}(\rho)$ preserves the column indices (from $P$) of its entries; therefore $\rho_i$ is the topmost entry in column $k-1$ of $P_{j-1}(\rho)$.
Thus since $\rho_i > \rho_j$, one obtains $P_j(\rho)$ by first inserting $\rho_j$ into some column strictly to the left of column $k$.
It follows (from the RS algorithm) that entry $j$ in $Q_j(\rho)$ also occurs strictly to the left of column $k$, and thus it is impossible that $Q_n(\rho) = Q$.
Thus we have $\rho \neq \sigma \coloneqq \operatorname{RS}^{-1}(P,Q)$.
Since $P = \rho \cdot Q^\uparrow$ and $\rho \neq \sigma$, we conclude that $P \neq \sigma \cdot Q^\uparrow$, which completes the proof of the converse.
\end{proof}

\begin{corollary}
    \label{corollary:MM column-canonical}
    Let $\lambda \vdash n$, and let $Q \in \SYT(\lambda)$ be admissible.
    Let $\sigma \in S_n$, and let $P = \sigma \cdot Q^\uparrow$.
    If $P \in \SYT(\lambda)$, then $\operatorname{RS}(\sigma) = (P,Q)$, and in particular, we have $\sh(\sigma) = \lambda$.
\end{corollary}

The tableau $\Tdown$ defined in~\eqref{Tdown} is, in a sense, the canonical example of an admissible tableau; note that admissibility is automatic because each entry in $\Tdown$ is less than every entry in the column to its right.
The following construction is a sort of ``proto-$\alpha$-coloring'' of $\Tup$ (in the sense of the upcoming Definition~\ref{def: alpha coloring}).
It gives a canonical way to cyclically permute the entries in $\Tup$ in order to obtain a standard tableau (see Figure~\ref{fig: canonical cycle example} for an example).

\begin{construction}
\label{constr:canonical cycle}
    Let $\lambda \in \mathcal{B}_{(n)}$.
    The \emph{canonical cycle of shape~$\lambda$} is the permutation $\sigma \in S_n$ constructed as follows:
    \begin{enumerate}
        \item Begin with the tableau $\Tup$, where $\Tdown$ is the canonical tableau defined in~\eqref{Tdown}.
        
        \item Draw a slash through the bottom box in each column, except for the rightmost column and any columns of length $1$.
        
        In each column (starting from the left), draw an arrow from the top box to the bottommost (non-slashed) box, then alternately to the top/bottom remaining boxes until the column is exhausted.
        When there are no more boxes in the column, draw the next arrow up to the top box in the next column to the right, and repeat the above procedure in that column.
        
        When the boxes in the rightmost column are exhausted, complete the cycle by drawing a sequence of arrows to/from the slashed boxes, from right to left, and then a final arrow up to the top box in the first column.
        
        \item To obtain $\sigma$, view each arrow as giving the preimage $\sigma_i \mapsto i$.
        Hence to write $\sigma$ in cycle notation, simply follow the arrows \emph{backwards} (starting anywhere) and write down the corresponding entries.
        \end{enumerate}
\end{construction} 


\begin{figure}
    \centering
    \begin{subfigure}[b]{0.3\textwidth}
         \centering
         \input{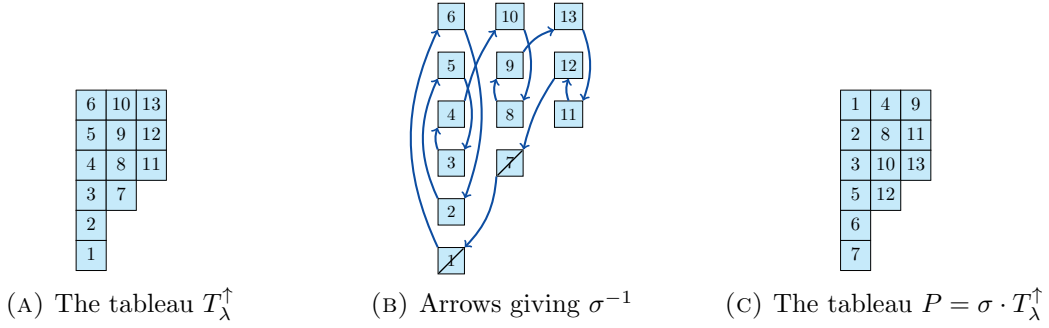}
         \caption{The tableau $\Tup$}
         \label{subfig:Qupex}
     \end{subfigure}
    \begin{subfigure}[b]{0.3\textwidth}
         \centering
         \scalebox{.65}{
\begin{tikzpicture}[scale=1, node distance=0.5cm, every node/.style={draw, minimum size=0.6cm, anchor=center, fill=\Blue, scale=.9}]

\node (6) at (0,0) {6};
\node (10) at (1.2,0) {10};
\node (13) at (2.4,0) {13};

\node (5) at (0,-1) {5};
\node (9) at (1.2,-1) {9};
\node (12) at (2.4,-1) {12};

\node (4) at (0,-2) {4};
\node (8) at (1.2,-2) {8};
\node (11) at (2.4,-2) {11};

\node (3) at (0,-3) {3};
\node (7) at (1.2,-3) {7};

\node (2) at (0,-4) {2};

\node (1) at (0,-5) {1};

\draw [thick] (-0.25,-5.25) -- ++(.5,.5);

\draw [thick] (0.95,-3.25) -- ++(.5,.5);

\draw[->, very thick, cyan!20!blue, bend left=20] (6.south east) to (2.north east);
\draw[->, very thick, cyan!20!blue, bend left=25] (2.north west) to (5.south west);
\draw[->, very thick, cyan!20!blue, bend left=20] (5.south east) to (3.north east);
\draw[->, very thick, cyan!20!blue, bend left=20] (3.north west) to (4.south west);
\draw[->, very thick, cyan!20!blue, bend left=10] (4.north east) to (10.south west);
\draw[->, very thick, cyan!20!blue, bend left=20] (10.south east) to (8.north east);
\draw[->, very thick, cyan!20!blue, bend left=20] (8.north west) to (9.south west);
\draw[->, very thick, cyan!20!blue, bend left=20] (9.north east) to (13.south west);
\draw[->, very thick, cyan!20!blue, bend left=20] (13.south east) to (11.north east);
\draw[->, very thick, cyan!20!blue, bend left=20] (11.north) to (12.south);
\draw[->, very thick, cyan!20!blue, bend right=10] (12.south west) to (7.north east);
\draw[->, very thick, cyan!20!blue, bend left=20] (7.south west) to (1.north east);
\draw[->, very thick, cyan!20!blue, bend left=25] (1.north west) to (6.south west); 

\end{tikzpicture}
}
         \caption{Arrows giving $\sigma^{-1}$}
         \label{subfig:arrows}
     \end{subfigure}
     \begin{subfigure}[b]{0.3\textwidth}
         \centering
         \input{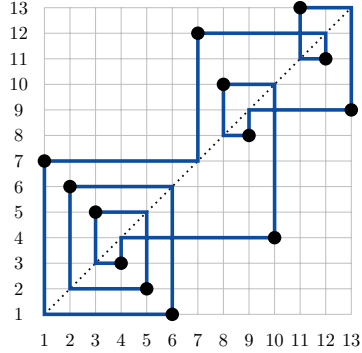}
         \caption{The tableau $P = \sigma \cdot \Tup$}
         \label{subfig:Pex}
     \end{subfigure}
    
    \vspace{1cm}
     
     \begin{subfigure}[b]{0.75\textwidth}
         \centering
         \scalebox{.85}{\begin{tikzpicture}[scale=.4, every node/.style={scale=.75}]
        \draw[lightgray] (1,1) grid (13,13);
        \draw[dotted,thick] (1,1) -- (13,13);


        \draw [ultra thick, cyan!20!blue] (1,7) node [dot,fill=black] {} -- (7,7) -- (7,12) node [dot,fill=black] {} -- (12,12) -- (12,11) node [dot,fill=black] {} -- (11,11) -- (11,13) node [dot,fill=black] {} -- (13,13) -- (13,9) node [dot,fill=black] {} -- (9,9) -- (9,8) node [dot,fill=black] {} -- (8,8) -- (8,10) node [dot,fill=black] {} -- (10,10) -- (10,4) node [dot,fill=black] {} -- (4,4) -- (4,3) node [dot,fill=black] {} -- (3,3) -- (3,5) node [dot,fill=black] {} -- (5,5) -- (5,2) node [dot,fill=black] {} -- (2,2) -- (2,6) node [dot,fill=black] {} -- (6,6) -- (6,1) node [dot,fill=black] {} -- (1,1) -- (1,7);

        \node at (1,0) {$1$};
        \node at (2,0) {$2$};
        \node at (3,0) {$3$};
        \node at (4,0) {$4$};
        \node at (5,0) {$5$};
        \node at (6,0) {$6$};
        \node at (7,0) {$7$};
        \node at (8,0) {$8$};
        \node at (9,0) {$9$};
        \node at (10,0) {$10$};
        \node at (11,0) {$11$};
        \node at (12,0) {$12$};
        \node at (13,0) {$13$};

        \node at (0,1) {$1$};
        \node at (0,2) {$2$};
        \node at (0,3) {$3$};
        \node at (0,4) {$4$};
        \node at (0,5) {$5$};
        \node at (0,6) {$6$};
        \node at (0,7) {$7$};
        \node at (0,8) {$8$};
        \node at (0,9) {$9$};
        \node at (0,10) {$10$};
        \node at (0,11) {$11$};
        \node at (0,12) {$12$};
        \node at (0,13) {$13$};

\end{tikzpicture}
}
         \caption{The graph of $\sigma$, with its cycle structure shown by the spiraling path.
         Note that this path (upon reflection about the main diagonal) gives the same pattern as the arrows in (\textsc{b}).}
         \label{subfig:graph}
     \end{subfigure}
    \caption{Illustration of Construction~\ref{constr:canonical cycle}: the canonical cycle $\sigma$ of shape $\lambda = (3,3,3,2,1,1)$.
    The arrows give the cycle notation of $\sigma^{-1} =  (1,6,2,5,3,4,10,8,9,13,11,12,7)$.
    The one-line notation of $\sigma$ is the reverse column word of $P = \sigma \cdot \Tup$ (i.e., the tableau obtained from $\Tup$ by permuting entries along the arrows): thus we have
    $\sigma = [7,6,5,3,2,1,12,10,8,4,13,11,9]$.
    Via the RS correspondence, we have $\operatorname{RS}(\sigma) = (P, \Tdown)$.}
    \label{fig: canonical cycle example}
\end{figure}

The arrows in Construction~\ref{constr:canonical cycle} can be viewed as a condensed depiction of the graph of the canonical cycle $\sigma$ (see Figure~\ref{subfig:graph}).
Indeed, each column of $\Tup$ represents a cross section of a ``spiral'' pattern in this graph: since the arrows actually depict $\sigma^{-1}$, one can see the precise structure of the arrows by reflecting the graph of $\sigma$ about the main diagonal.
The term ``canonical cycle of shape $\lambda$'' is justified by the following lemma.

\begin{lemma}
    \label{lemma:canonical cycle}
    Let $\lambda \in \mathcal{B}_{(n)}$, and let $\sigma$ be the canonical cycle of shape $\lambda$ obtained via Construction~\ref{constr:canonical cycle}.
    Then we have ${\sigma \in \mathcal{C}_{(n)}}$, and  $\sh(\sigma) = \lambda$.
    Moreover, we have $\operatorname{RS}(\sigma) = (P,\Tdown)$, where $P = \sigma \cdot \Tup$ is the tableau whose reverse column word is the one-line notation of $\sigma$.
\end{lemma}

\begin{proof}
    The fact that $\sigma \in \mathcal{C}_{(n)}$ is immediate from Construction~\ref{constr:canonical cycle}.
    Since $\Tdown$ is admissible, the rest of the lemma will follow from Corollary~\ref{corollary:MM column-canonical} if we can show that $P \in \SYT(\lambda)$.
        
    To show that the entries of $P$ increase down each column, we first observe that in a given column of $P$, all except perhaps the top and bottom entry were originally positioned in that same column of $\Tup$; therefore, since the arrows with both ends in this column simply reverse the order of the corresponding entries (which were decreasing downwards in $\Tup$), we conclude that the entries strictly between the top and bottom entries are now \emph{increasing} downwards in $P$.
    In any column of $P$ with more than one entry, its bottommost entry came from its right in $\Tup$, except that in the last column of $P$ the bottom entry is $n$.
    Likewise, in any column of $P$ with more than one entry,
    its topmost entry came from its left in $\Tup$, except that in the first column the top entry is $1$.
    (This last statement assumes that $\Tup$ has more than one column, which (when $n \geq 3$) requires the hypothesis $\lambda \in \mathcal{B}_{(n)}$.
    Note that when $n=1$ or $n=2$, there is exactly one shape $\lambda \in \mathcal{B}_{(n)}$, which is a single column with 1 or 2 boxes, respectively;
    these are the only one-column shapes for which the arrows defined in Construction~\ref{constr:canonical cycle} give a valid cycle.)
    Hence the top and bottom entries in $P$ are the least and greatest entries (respectively) in their columns.
    This proves that the entries of $P$ increase down columns.

    The same argument also shows that the entries in $P$ increase across each row (from left to right): in a given row, any entry except the \emph{top} or \emph{bottom} of a column originally began in that column of $\Tup$, so these entries increase from left to right.
    We will treat the top row of $P$ in the last paragraph of the proof.
    If an entry is the \emph{bottom} of a column, then the remaining entries to its right in its row are also at the bottom of a column; each of these entries (since it is slashed) came from the column to its right in  $\Tup$ (and thus is larger than the entry to its left in $P$), the only exception being the bottom entry of the rightmost column, which is either $n$ (if that column contains more than one entry) or else lies in the top row of $P$ (to be discussed next).
    
    Finally, in the top row of $P$, every entry came from the preceding column in $\Tup$, except for the leftmost entry, which is $1$ (assuming that the first column contains at least two entries, which requires the hypothesis $\lambda \in \mathcal{B}_{(n)}$); hence the top row of $T$ is increasing.
    This proves that $P \in \SYT(\lambda)$, and the result follows from Corollary~\ref{corollary:MM column-canonical}.
\end{proof}

Note that our Construction~\ref{constr:canonical cycle} immediately yields one of the main results of~\cite{RS-complete-Rubinstein}, which we reproduced above in part (1) of Theorem~\ref{thm: original paper main result}.
This result is the following corollary (restated in the language of the present paper).

\begin{corollary}
    We have $\mathcal{S}_{(n)} = \mathcal{B}_{(n)}$ for all $n \in \N$.
\end{corollary}


\section{Constructions of $\alpha$-colorings}
\label{sec:alpha-colorings}

We will now extend Construction~\ref{constr:canonical cycle} to more general cycle types $\alpha = (\alpha_1, \alpha_2) \vdash n$.
Guided by Corollary~\ref{corollary:MM column-canonical}, we can immediately see that the goal is to find \emph{two} disjoint cycles of arrows (of lengths $\alpha_1$ and $\alpha_2$) in the tableau $\Tup$, such that permuting the entries along the arrows yields a standard Young tableau $P$.
In fact, if we fix the method of drawing arrows described in Construction~\ref{constr:canonical cycle}, then the goal reduces even further: namely, decompose $\Tup$ into two disjoint subdiagrams (i.e., ``colors'') of sizes $\alpha_1$ and $\alpha_2$, such that permuting entries along the arrows (restricted to each color) yields a standard tableau $P$.
This goal leads to the following general definition of an $\alpha$-coloring of a tableau.

\begin{definition}
\label{def: alpha coloring}
    Let $\alpha = (\alpha_1, \ldots, \alpha_r) \vdash n$.
    Let $\lambda \in \mathcal{B}_{\alpha}$, and let $Q \in \SYT(\lambda)$ be admissible.
    An \emph{$\alpha$-coloring of $Q^\uparrow$} is a partitioning of the boxes of $Q^\uparrow$ into the colors $c_1, \ldots, c_r$, such that 
    \begin{enumerate}
        \item there are exactly $\alpha_i$ boxes with color $c_i$;
        \item upon cyclically permuting the entries of each color via the arrows defined in Step 2 of Construction~\ref{constr:canonical cycle}, while ignoring the boxes of other colors, one obtains a standard tableau $P \in \SYT(\lambda)$.
        (See the arrows in Figure \ref{fig: alpha-coloring example}.)
    \end{enumerate} 
    The \emph{associated permutation} of an $\alpha$-coloring is the permutation $\sigma$ such that $P = \sigma \cdot Q^\uparrow$.
\end{definition}

\begin{figure}
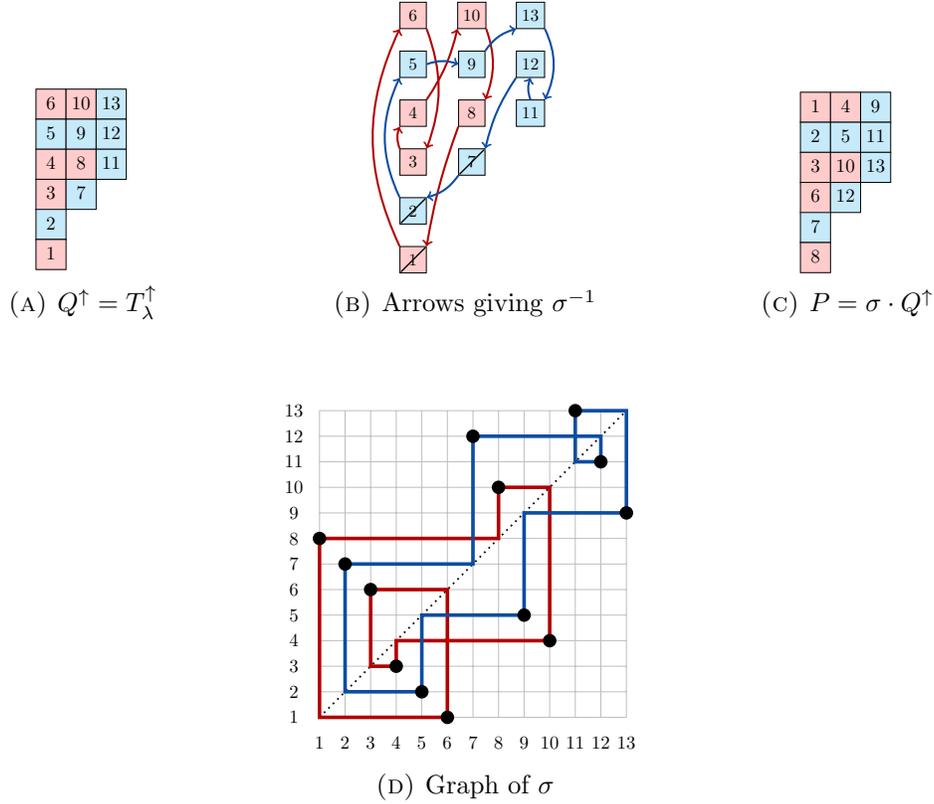

    \centering
    \begin{subfigure}[b]{0.3\textwidth}
         \centering
         \input{Qup_alpha_example}
         \caption{$Q^\uparrow = \Tup$}
         \label{subfig:alpha Qupex}
     \end{subfigure}
    \begin{subfigure}[b]{0.3\textwidth}
         \centering
         \scalebox{.65}{
\begin{tikzpicture}[scale=1, node distance=0.5cm, every node/.style={draw, minimum size=0.6cm, anchor=center, fill=\Blue, scale=.9}]

\node [fill=\Red] (6) at (0,0) {6};
\node [fill=\Red] (10) at (1.2,0) {10};
\node (13) at (2.4,0) {13};

\node (5) at (0,-1) {5};
\node (9) at (1.2,-1) {9};
\node (12) at (2.4,-1) {12};

\node [fill=\Red] (4) at (0,-2) {4};
\node [fill=\Red] (8) at (1.2,-2) {8};
\node (11) at (2.4,-2) {11};

\node [fill=\Red] (3) at (0,-3) {3};
\node (7) at (1.2,-3) {7};

\node (2) at (0,-4) {2};

\node [fill=\Red] (1) at (0,-5) {1};

\draw [thick] (-0.25,-5.25) -- ++(.5,.5);

\draw [thick] (-0.25,-4.25) -- ++(.5,.5);

\draw [thick] (0.95,-3.25) -- ++(.5,.5);

\draw[->, red!70!black, very thick, bend left=20] (6.south east) to (3.north east);
\draw[->, red!70!black, very thick, bend left=20] (3.north west) to (4.south west);
\draw[->, red!70!black, very thick, bend right=10] (4.north east) to (10.south west);
\draw[->, red!70!black, very thick, bend left=20] (10.south east) to (8.north east);
\draw[->, red!70!black, very thick, bend right=5] (8.south west) to (1.north east);
\draw[->, red!70!black, very thick, bend left=25] (1.north west) to (6.south west);

\draw[->, cyan!20!blue, very thick, bend left=25] (2.north west) to (5.south west);
\draw[->, cyan!20!blue, very thick, bend left=20] (5.east) to (9.west);
\draw[->, cyan!20!blue, very thick, bend left=20] (9.north east) to (13.south west);
\draw[->, cyan!20!blue, very thick, bend left=25] (13.south east) to (11.north east);
\draw[->, cyan!20!blue, very thick, bend left=20] (11.north) to (12.south);
\draw[->, cyan!20!blue, very thick, bend right=10] (12.south west) to (7.north east);
\draw[->, cyan!20!blue, very thick, bend left=20] (7.south west) to (2.north east);

\end{tikzpicture}
}
         \caption{Arrows giving $\sigma^{-1}$}
         \label{subfig:alpha arrows}
     \end{subfigure}
     \begin{subfigure}[b]{0.3\textwidth}
         \centering
         \input{P_alpha_ex}
         \caption{$P = \sigma \cdot Q^\uparrow$}
         \label{subfig:alpha Pex}
     \end{subfigure}
    
    \vspace{1cm}
     
     \begin{subfigure}[b]{0.75\textwidth}
         \centering
         \scalebox{.85}{
\begin{tikzpicture}[scale=.4, every node/.style={scale=.75}]
        \draw[lightgray] (1,1) grid (13,13);
        \draw[dotted,thick] (1,1) -- (13,13);

        \draw [ultra thick, red!70!black] (1,8) node [dot,fill=black] {} -- (8,8) -- (8,10) node [dot,fill=black] {} -- (10,10) -- (10,4) node [dot,fill=black] {} -- (4,4) -- (4,3) node [dot,fill=black] {} -- (3,3) -- (3,6) node [dot,fill=black] {} -- (6,6) -- (6,1) node [dot,fill=black] {} -- (1,1) -- (1,8);
        
        \draw [ultra thick, cyan!20!blue] (2,7) node [dot,fill=black] {} -- (7,7) -- (7,12) node [dot,fill=black] {} -- (12,12) -- (12,11) node [dot,fill=black] {} -- (11,11) -- (11,13) node [dot,fill=black] {} -- (13,13) -- (13,9) node [dot,fill=black] {} -- (9,9) -- (9,5) node [dot,fill=black] {} -- (5,5) -- (5,2) node [dot,fill=black] {} -- (2,2) -- (2,7);

        \node at (1,0) {$1$};
        \node at (2,0) {$2$};
        \node at (3,0) {$3$};
        \node at (4,0) {$4$};
        \node at (5,0) {$5$};
        \node at (6,0) {$6$};
        \node at (7,0) {$7$};
        \node at (8,0) {$8$};
        \node at (9,0) {$9$};
        \node at (10,0) {$10$};
        \node at (11,0) {$11$};
        \node at (12,0) {$12$};
        \node at (13,0) {$13$};

        \node at (0,1) {$1$};
        \node at (0,2) {$2$};
        \node at (0,3) {$3$};
        \node at (0,4) {$4$};
        \node at (0,5) {$5$};
        \node at (0,6) {$6$};
        \node at (0,7) {$7$};
        \node at (0,8) {$8$};
        \node at (0,9) {$9$};
        \node at (0,10) {$10$};
        \node at (0,11) {$11$};
        \node at (0,12) {$12$};
        \node at (0,13) {$13$};

\end{tikzpicture}
}
         \caption{Graph of $\sigma$}
         \label{subfig:alpha graph}
     \end{subfigure}
    \caption{Example of an $\alpha$-coloring of $Q^\uparrow = \Tup$, where $\alpha = (7,6)$ and $\lambda = (3,3,3,2,1,1)$.
    The coloring itself (7 \blue and 6 \red boxes) is shown in~\textsc{(a)}.
    The slashes and arrows in~\textsc{(b)} are given by Construction~\ref{constr:canonical cycle} applied to each color separately.
    The arrows determine the associated permutation $\sigma$, where $\sigma^{-1} = (1,6,3,4,10,8)(2,5,9,13,11,12,7)$.
    By permuting entries along the arrows, one obtains the tableau $P$ in~\textsc{(c)}, and the fact that $P$ is standard makes the coloring a true $\alpha$-coloring (by Definition~\ref{def: alpha coloring}).}
    \label{fig: alpha-coloring example}
\end{figure}

\begin{lemma}
\label{lemma:assoc perm}
Let $\alpha, \lambda \vdash n$ and let $Q \in \SYT(\lambda)$ be admissible.
If there exists an $\alpha$-coloring of $Q^\uparrow$, then $\lambda \in \mathcal{S}_\alpha$.
\end{lemma}

\begin{proof}
    Let $\sigma$ be the associated permutation of an $\alpha$-coloring of $Q^\uparrow$.
    Then we have $\sigma \in \mathcal{C}_\alpha$, since $\sigma$ is given by the arrows in Definition~\ref{def: alpha coloring}, which form disjoint cycles of lengths $\alpha_1, \ldots, \alpha_r$.
    Moreover, since $P = \sigma \cdot Q^\uparrow$ is standard (by Definition~\ref{def: alpha coloring}), Corollary~\ref{corollary:MM column-canonical} implies that ${\rm RS}(\sigma) = (P, Q)$, and in particular $\sh(\sigma) = \lambda$.
    Thus we have both $\sigma \in \mathcal{C}_\alpha$ and $\sh(\sigma) = \lambda$, which by definition~\eqref{S_alpha} implies that $\lambda \in \mathcal{S}_\alpha$.
\end{proof}

In the special case where $Q = \Tdown$ in Lemma~\ref{lemma:assoc perm} (which will usually be the case in our constructions), we can say even more: in this case, the one-line notation of $\sigma$ is precisely the reverse column word of $P = \sigma \cdot Q^\uparrow$.
In fact, Construction~\ref{constr:canonical cycle} was simply the special case of an $\alpha$-coloring of $\Tup$ where $r=1$.
(In other words, Lemma~\ref{lemma:assoc perm} generalizes Lemma~\ref{lemma:canonical cycle} for $r \geq 1$.)
Unlike the $r=1$ case, where the $\alpha$-coloring was unique (there being only one color), for $r > 1$ there may be several $\alpha$-colorings of a given tableau $Q^\uparrow$, each determining a different $\sigma \in \mathcal{C}_\alpha$ such that $\sh(\sigma) = \lambda$.

For our immediate purposes where $r=2$, we will refer to the two colors $c_1$ and $c_2$ as \emph{\blue} and \emph{\red}, respectively.
(See the example in Figure~\ref{fig: alpha-coloring example}.)
In the rest of this section, we will effectively prove one direction of Theorem~\ref{thm:main result} by constructing $\alpha$-colorings that apply to every shape $\lambda \in \mathcal{B}_\alpha$ except for the table of exceptions given in the theorem.
(There is one sporadic family of pairs $(\alpha,\lambda)$ where $\lambda \in \mathcal{S}_\alpha$ but no $\alpha$-coloring exists; in this case, we will exhibit the requisite $\sigma$ directly.)
By Lemma~\ref{lemma:assoc perm}, the existence of each such $\alpha$-coloring implies that $\lambda \in \mathcal{S}_\alpha$. 
The main ingredient of our $\alpha$-colorings will be the following scheme which we call a \emph{spiral coloring} of a column (see examples in Figure~\ref{fig:spiral coloring}).

 \begin{figure}[t]
    \centering
    \begin{subfigure}[b]{0.4\textwidth}
         \centering
         \scalebox{.75}{
\begin{tikzpicture}[scale=.5, node distance=0, every node/.style={draw, minimum size=0.7cm, scale=.7, anchor=center}]

\node [fill=\Red] (10) at (0,10) {10};
\node [fill=\Blue] (9) at (0,9) {9};
\node [fill=\Red] (8) at (0,8) {8};
\node [fill=\Red] (7) at (0,7) {7};
\node [fill=\Red] (6) at (0,6) {6};
\node [fill=\Red] (5) at (0,5) {5};
\node [fill=\Blue] (4) at (0,4) {4};
\node [fill=\Red] (3) at (0,3) {3};
\node [fill=\Blue] (2) at (0,2) {2};
\node [fill=\Red] (1) at (0,1) {1};

\draw[->, very thick, red!70!black,  bend left=60, shorten >= 5 pt, shorten <= 5 pt] (1.west) to (10.west);
\draw[->, very thick, red!70!black,  bend left=60, shorten >= 5 pt, shorten <= 5 pt] (10.east) to (3.east);
\draw[->, very thick, red!70!black,  bend left=60, shorten >= 5 pt, shorten <= 5 pt] (3.west) to (8.west);
\draw[->, very thick, red!70!black,  bend left=60, shorten >= 5 pt, shorten <= 5 pt] (8.east) to (5.east);
\draw[->, very thick, red!70!black,  bend left=60, shorten >= 5 pt, shorten <= 5 pt] (5.west) to (7.west);
\draw[->, very thick, red!70!black,  bend left=90, shorten >= 2 pt, shorten <= 2 pt] (7.east) to (6.east);

\draw[->, very thick, cyan!70!black,  bend left=60, shorten >= 5 pt, shorten <= 5 pt] (2.west) to (9.west);
\draw[->, very thick, cyan!70!black,  bend left=60, shorten >= 5 pt, shorten <= 5 pt] (9.east) to (4.east);

\end{tikzpicture}
}
         \caption{Spiral coloring with $3$ \blue boxes and $7$ \red boxes, where \red is on the outside.
        The \blue terminal box has the entry 4, and the \red terminal box has the entry 6.}
     \end{subfigure}
     \qquad
     \begin{subfigure}[b]{0.4\textwidth}
         \centering
         \scalebox{.75}{
\begin{tikzpicture}[scale=.5, node distance=0, every node/.style={draw, minimum size=0.7cm, scale=.7, anchor=center}]

\node [fill=\Blue] (10) at (0,10) {10};
\node [fill=\Red] (9) at (0,9) {9};
\node [fill=\Blue] (8) at (0,8) {8};
\node [fill=\Red] (7) at (0,7) {7};
\node [fill=\Red] (6) at (0,6) {6};
\node [fill=\Blue] (5) at (0,5) {5};
\node [fill=\Red] (4) at (0,4) {4};
\node [fill=\Blue] (3) at (0,3) {3};
\node [fill=\Red] (2) at (0,2) {2};
\node [fill=\Blue] (1) at (0,1) {1};

\draw[->, very thick, red!70!black,  bend left=60, shorten >= 5 pt, shorten <= 5 pt] (2.west) to (9.west);
\draw[->, very thick, red!70!black,  bend left=60, shorten >= 5 pt, shorten <= 5 pt] (9.east) to (4.east);
\draw[->, very thick, red!70!black,  bend left=60, shorten >= 5 pt, shorten <= 5 pt] (4.west) to (7.west);
\draw[->, very thick, red!70!black,  bend left=60, shorten >= 2 pt, shorten <= 2 pt] (7.east) to (6.east);

\draw[->, very thick, cyan!70!black,  bend left=60, shorten >= 5 pt, shorten <= 5 pt] (1.west) to (10.west);
\draw[->, very thick, cyan!70!black,  bend left=60, shorten >= 5 pt, shorten <= 5 pt] (10.east) to (3.east);
\draw[->, very thick, cyan!70!black,  bend left=60, shorten >= 5 pt, shorten <= 5 pt] (3.west) to (8.west);
\draw[->, very thick, cyan!70!black,  bend left=60, shorten >= 5 pt, shorten <= 5 pt] (8.east) to (5.east);

\end{tikzpicture}
}
         \caption{Spiral coloring with $5$ \blue boxes and $5$ \red boxes, where \blue is on the outside.
        The \blue terminal box has the entry 5, and the \red terminal box has the entry 6.}
     \end{subfigure}
    \caption{Examples of spiral colorings (Definition~\ref{def:spiral coloring}).
    For these examples we color the first column of $\Tup$.}
    \label{fig:spiral coloring}
\end{figure}
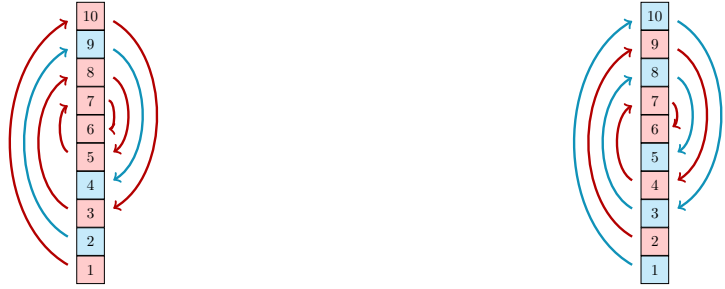

\begin{definition}[Spiral colorings and terminal boxes]
    \label{def:spiral coloring}
    Given a column in a tableau, and a known number of \blue and \red boxes (summing to the length of the column), we give the column the \emph{spiral coloring} (with a designated color \emph{on the outside}), as follows:
\begin{itemize}
    \item For each color alternately (starting with the designated \emph{outside color}), color the bottommost available box, followed by the topmost available box.
    \item As soon as one color is exhausted, color the remaining boxes with the opposite color.
\end{itemize}
In a spiral coloring, we define the \emph{terminal box} of each color to be the last box that is given that color in the process described above. 
\end{definition}

The term \emph{spiral coloring} is due to the fact that the coloring can be viewed as the cross section obtained by interlacing a \blue and a \red spiral (both directed clockwise, the outer spiral having the outside color), and each \emph{terminal box} is indeed the endpoint of a spiral.
In fact, each of these spirals is precisely the spiral of arrows in the first tableau column in Construction~\ref{constr:canonical cycle} (upon restricting one's attention to the boxes of a single color).
Thus whenever we give the first column of $Q^\uparrow$ a spiral coloring, the entire column (with the exception of the bottommost two boxes) automatically satisfies the condition for an $\alpha$-coloring in Definition~\ref{def: alpha coloring}.
For the bottommost two boxes (which are slashed in Construction~\ref{constr:canonical cycle}), the permuted entries are taken from some column to the right, and are thus greater than all the entries above them; in our proofs, we therefore need only show that these two bottommost entries are in increasing order.

We give a brief overview of our constructions of $\alpha$-colorings.
Given $\alpha=(\alpha_1, \alpha_2)$ and $\lambda$, we specify an admissible tableau $Q \in \SYT(\lambda)$ (usually $Q = \Tdown$); then in the tableau $Q^\uparrow$, we color $\alpha_1$ many boxes \blue, and $\alpha_2$ many boxes \red. 
As a guiding principle, we generally color boxes \red starting from the left, and \blue starting from the right.
Typically there is a  column in which both colors must appear, for which we specify a spiral coloring to obtain the desired result.
In certain edge cases, such as when there are too few \red boxes to escape the first column, or too few \blue boxes to escape the tail, it is also necessary to mix colors in a second column or along the tail.
(There is also one case, namely case (1a) of Construction~\ref{constr:2-column shapes}, where no $\alpha$-coloring exists; in this case we exhibit $\sigma \in \mathcal{C}_\alpha$ such that $\sh(\sigma) = \lambda$.) 
After each construction, we state a lemma verifying that the given coloring is in fact an $\alpha$-coloring; in each case, the proof amounts to showing that $\sigma \cdot Q^\uparrow$ is a standard tableau, where $\sigma \in \mathcal{C}_\alpha$ is the associated permutation of the $\alpha$-coloring (Definition~\ref{def: alpha coloring}).

\begin{tcolorbox}[parbox=false,breakable,enhanced
]
    \textbf{Outline: constructions of $\alpha$-colorings}

    \tcblower

    Let $\alpha=(\alpha_1,\alpha_2)\vdash n$ and $\lambda \in \mathcal{B}_\alpha$, excluding the shapes $\lambda$ listed in the table in Theorem~\ref{thm:main result}.
    For a specified admissible tableau $Q \in \SYT(\lambda)$, we construct an $\alpha$-coloring of $Q^\uparrow$, determined by the following conditions on~$\alpha$ and~$\lambda$:

\begin{itemize}
    \item $\alpha_2\geq 3$, and $\lambda$ has at least three columns ($\lambda_1 \geq 3$), with at least two boxes in the second column ($\lambda_2 \geq 2$); exclude the special case $\alpha_2=\lambda'_1 = 3$:

        \begin{itemize}
            \item Colors meet in the middle ($\lambda'_1 < \alpha_2 < n - \lambda_1 + \lambda_2$): Construction~\ref{constr:nice}
            
            \item Red boxes fit in the first column ($\alpha_2 \leq \lambda'_1$): Construction~\ref{constr:tricky left}
            
            \item Blue boxes fit in the tail ($\alpha_2 \geq n - \lambda_1 + \lambda_2$): Construction~\ref{constr:long tails} 
        \end{itemize}
    \item $\alpha_2\geq 3$, remaining cases:
    
        \begin{itemize}
            \item $\alpha_2=3$ and $\lambda$ has three rows ($\lambda_1'=3$): Construction~\ref{constr:alpha2=colambda1=3}
            \item Remaining cases where $\lambda$ has two columns ($\lambda_1 = 2$): Construction~\ref{constr:2-column shapes}
            \item Remaining cases where $\lambda$ is a hook ($\lambda_2 = 1$): Construction~\ref{constr:hooks}
        \end{itemize}
    \item $\alpha_2 = 2$: Construction~\ref{constr:2}
    \item $\alpha_2 = 1$: Construction~\ref{constr:1}
\end{itemize}

\end{tcolorbox}

\ytableausetup{smalltableaux}

\begin{construction}
\label{constr:nice}
    Here we assume that $\alpha_2 \geq 3$, that $\lambda$ has at least three columns and is not a hook, and that $\lambda'_1 < \alpha_2 < n - \lambda_1 + \lambda_2$.
    Take $Q = \Tdown$, and color $\Tup$ as follows:
    
    From the left (resp., right) side, color $\alpha_2$ many boxes \red (resp., $\alpha_1$ many boxes \blue) until the leftover boxes of each color meet in a single ``mixing'' column.
    (If there are no leftovers, then the coloring is complete.)
    In the mixing column, color the top box 
    \red.
        For the rest of the mixing column, use the spiral coloring with \blue on the outside. See below for an example.
        
    \[
    \alpha = (8,7): \quad \begin{ytableau}[*(\Blue)]
        *(\Red) 5 & *(\Red) 9 & 13 & 14 & 15\\
        *(\Red)4 & 8 & 12\\
        *(\Red)3 & *(\Red)7 & 11\\
        *(\Red)2 & 6 & 10\\
        *(\Red)1
    \end{ytableau}
    \qquad\qquad
    \sigma^{-1} = (1,5,2,4,3,9,7)(6,8,13,11,12,14,15,10)
    \]
    
\end{construction}

\begin{lemma}
    \label{lemma:nice}
    Construction~\ref{constr:nice} produces an $\alpha$-coloring of $\Tup$.
\end{lemma}

\begin{proof}
    If there is no mixing column, then the result follows directly from Lemma~\ref{lemma:canonical cycle} applied to each of the two subdiagrams of each color, since every entry in the \blue subdiagram is greater than
    every entry in the \red subdiagram.

    Hence, suppose that there is a mixing column.
    By assumption, the mixing column is neither the first nor last column of $\Tup$.
    Note that (in the context of Construction~\ref{constr:canonical cycle}) the only slashed box in the mixing column is the \blue box on bottom (unless it is the only \blue box in the column).
    All except the top and bottom entries in this column of $\sigma \cdot \Tup$ are obtained from entries in this same column of $\Tup$, and the spiral coloring guarantees that they increase downwards.
    The top entry of the mixing column comes from the \red column to its left, while the bottom entry comes from the \blue column to its right, and so these are the least and greatest entries in the column, meaning that the entire mixing column is standard in $\sigma \cdot \Tup$.
    By the same argument, every entry in the mixing column in $\sigma \cdot \Tup$ is greater than the entry to its left and less than the entry to its right.
    The remaining subdiagrams of each color are standard by Lemma~\ref{lemma:canonical cycle}, and thus the entire tableau $\sigma \cdot \Tup$ is standard.
\end{proof}

\begin{construction}
    \label{constr:tricky left}
    Here we assume that $\alpha_2 \geq 3$, that $\lambda$ has at least three columns and is not a hook, and that $\alpha_2 \leq \lambda'_1$.
    We omit the case $\alpha_2 = \lambda'_1 = 3$, which is handled in Construction~\ref{constr:alpha2=colambda1=3}.
    Take $Q = \Tdown$, and in $\Tup$ color every column to the right of column 2 entirely \blue.
    Then color columns 1 and 2 as follows:
    
    \begin{enumerate}
    \item If $\alpha_2 < \lambda'_1$, then there is exactly one \red box in column 2, specifically one of the two boxes on top.
    The color of the top box in column 2 matches the color of the lower terminal box in the spiral coloring of column 1:
    \begin{enumerate}
        \item If $\lambda'_2 = 2$, then give column 1 the spiral coloring with \blue on the outside.
        \item If $\lambda'_2 > 2$, then give column 1 the spiral coloring with \red on the outside.
    \end{enumerate}

    \item If $\alpha_2 = \lambda'_1$, then there are exactly two \red boxes in column 2.
        \begin{enumerate}
            \item If $\lambda'_2 \leq 3$, then give column 1 the spiral coloring with \blue on the outside.
            The two \red boxes in column 2 are the top and bottom boxes.
            \item If $\lambda'_2 > 3$, then give column 1 the spiral coloring with \red on the outside.
            In column 2, one of the \red boxes is the second box from the bottom.
            The other \red box is one of the top two boxes; the color of the top box is determined by the color of the lower terminal box in the spiral coloring of column 1.
        \end{enumerate}        
    \end{enumerate}
See below for a examples of each case.

    \[
\alpha = (8,7): \quad \begin{ytableau}[*(\Blue)]
     8 & *(\Red)10 & 12 & 14 & 15\\
     *(\Red)7 & 9 & 11 & 13\\
     *(\Red)6 \\
     *(\Red)5\\
     *(\Red)4\\
     *(\Red)3\\
     *(\Red)2\\
     1\\
     \none
\end{ytableau}
\qquad
\begin{ytableau}[*(\Blue)]
     *(\Red)9 & 12 & 14 & 15\\
     8 & *(\Red)11 & 13\\
     *(\Red)7 & 10 \\
     *(\Red)6\\
     *(\Red)5\\
     4\\
     *(\Red)3\\
     2\\
     *(\Red)1
\end{ytableau}
\qquad
\begin{ytableau}[*(\Blue)]
     7 & *(\Red)10 & 13 & 15\\
     *(\Red)6 & 9 & 12 & 14 \\
     *(\Red)5 & *(\Red)8 & 11\\
     *(\Red)4\\
     *(\Red)3\\
     *(\Red)2\\
     1\\
     \none\\
     \none
\end{ytableau}
\qquad
\begin{ytableau}[*(\Blue)]
     *(\Red)7 & *(\Red)11 & 13 & 14 & 15\\
     6 & 10 & 12\\
     *(\Red)5 & *(\Red)9\\
     *(\Red)4 & 8\\
     *(\Red)3\\
     2\\
     *(\Red)1\\
     \none\\
     \none
\end{ytableau}
\]

\end{construction}

\begin{lemma}
    \label{lemma:tricky left}
    Construction~\ref{constr:tricky left} produces  an $\alpha$-coloring of $\Tup$.
\end{lemma}

\begin{proof}\ 

\begin{enumerate}
    \item 
    \begin{enumerate} 
    \item In this case, there is only one \blue box in column 2; therefore the bottom entry in column 1 of $\sigma \cdot \Tup$, which is \blue (due to the spiral coloring with \blue on the outside), comes from some column to the right of column 2 in $\Tup$, and is therefore greater than the \red entry just above it, which came from column 2 of $\Tup$.
    This shows that column 1 of $\sigma \cdot \Tup$ is standard.
    (As usual, the standardness of the rest of $\sigma \cdot \Tup$ now easily follows from the nature of the spiral coloring, from matching the top box of column~2 with the lower terminal box of column~1, and from Lemma~\ref{lemma:canonical cycle} applied to each color.)

    \item In this case, there are at least two \blue boxes in column 2, at least one of which is below the \red box.
    Thus the bottommost entry in column 2 is \blue, and it appears at the bottom of column 1 in $\sigma \cdot \Tup$.
    Since this entry is greater than the \red entry in column 2 of $\Tup$, and since this \red entry appears just above the bottom entry in column 1 of $\sigma \cdot \Tup$, we conclude that column 1 of $\sigma \cdot \Tup$ is standard.
    The standardness of column 2 follows from our coloring of the topmost box.

    \end{enumerate}

    \item 

    \begin{enumerate}
        \item In this case, the only two \blue boxes in column 1 are the top and bottom boxes, and therefore the lower terminal box is \red. 
        Thus coloring the top box of column 2 \red guarantees that column 2 of $\sigma \cdot \Tup$ is standard.
        Since there is at most one \blue box in column 2, the same argument as in case (1a) guarantees the standardness of column 1 of $\sigma \cdot \Tup$.

        \item In this case, the standardness of columns 1 and 2 of $\sigma \cdot \Tup$ is proven by the same argument as in case (1b). \qedhere
    \end{enumerate}
\end{enumerate}
    
\end{proof}

\begin{construction}
\label{constr:long tails}
    Here we assume that $\alpha_2 \geq 3$, that $\lambda$ has at least three columns and is not a hook, and that $\alpha_2 \geq n + \lambda_1 - \lambda_2$ (in other words, the tail of $\lambda$ has length at least $\alpha_1$).
    Take $Q = \Tdown$, and color $\Tup$ as follows:

    In the column immediately left of the tail, color the bottom box \blue, and color the rest of the boxes to the left of the tail \red.
    The remaining \blue boxes are in the tail, where they alternate with \red boxes (starting with \blue, from left to right). See below for examples. 

    \[
    \alpha = (8,7): \qquad
    \begin{ytableau}
    [*(\Blue)]
    *(\Red)4 & *(\Red)7 & 8 & *(\Red)9 & 10 & 11 & 12 & 13 & 14 & 15\\
    *(\Red)3 & *(\Red)6\\
    *(\Red)2 & 5\\
    *(\Red)1
\end{ytableau}
\qquad 
    \begin{ytableau}[*(\Blue)]
     *(\Red)3 & *(\Red)5 & 6 & *(\Red)7 & 8 & *(\Red)9 & 10 & *(\Red)11 & 12 & 13 & 14 & 15\\
     *(\Red)2 & 4\\
     *(\Red)1\\
     \none
\end{ytableau}
\]
\end{construction}

\begin{lemma}
    \label{lemma:long tails}
    Construction~\ref{constr:long tails} produces an $\alpha$-coloring of $\Tup$.
\end{lemma}

\begin{proof}

We begin by claiming that the entry $n$ is \blue.
Since we assume that $\lambda$ is not a hook, there must be at least four boxes outside the tail.
Therefore there are at least four \red boxes in total, of which at least three are outside the tail, and at least one is inside the tail.
Thus we have
\begin{equation*}
\label{redtail inequality 1}
    1 \leq \#\{\text{\red boxes in tail}\} \leq \alpha_2 - 3.
\end{equation*}
But since $\alpha_2 \leq \alpha_1$, and $\alpha_1$ is one more than the number of \blue boxes in the tail, we have
\[
1 \leq \#\{\text{\red boxes in tail}\} \leq \#\{\text{\blue boxes in tail}\} - 2.
\]
This proves our claim that $n$ is \blue, since \blue and \red alternate in the tail from left to right until the \red boxes are exhausted.

To show that $\sigma \cdot \Tup$ is standard, we observe that there is at most one \blue box in each column, and so $\sigma$ permutes the \blue entries cyclically to the right, sending $n$ down to the \blue box which is at the bottom of the column preceding the tail (where the $n$ cannot possibly violate standardness in $\sigma \cdot \Tup$).
Meanwhile, $\sigma$ permutes the \red entries in the tail in a similar way.
Note that the leftmost \red entry in the tail of $\sigma \cdot \Tup$ is greater than the leftmost \blue entry in the tail, since this \red entry lies straight above this \blue entry in $\Tup$; the alternation of \blue and \red then guarantees that the entries in the tail of $\sigma \cdot \Tup$ are increasing from left to right.
The rest of $\sigma \cdot \Tup$ is standard by Lemma~\ref{lemma:canonical cycle}.
\end{proof}

\begin{construction}
    \label{constr:alpha2=colambda1=3}

    Here we assume that $\alpha_2 = 3$ and $\lambda'_1 = 3$.
    We exclude $\lambda = (4,1,1)$, since it occurs in the table in Theorem~\ref{thm:main result}.

    \begin{enumerate}
    \item If $\lambda'_2 = 1$, then let $Q = \Tdown$.
    The \red boxes in $\Tup$ have entries 2, 4, and~6.

    \item If $\lambda'_2 = 2$, then let $Q = \Tdown$.
    The \red boxes in $\Tup$ have entries 2, 4, and~5.

    \item If $\lambda'_2 = 3$, then let $Q$ be the tableau obtained from $\Tdown$ by transposing the entries 3 and 4.
    The \red boxes in  $Q^\uparrow$ have entries 1, 4, and~5.
    \end{enumerate}
See below for an example of each case.

\[
\alpha = (8,3):
\qquad
\begin{ytableau}
    [*(\Blue)]
    3 & *(\Red)4 & 5 & *(\Red)6 & 7 & 8 & 9 & 10 & 11\\
    *(\Red)2\\
    1
\end{ytableau}
\qquad
\begin{ytableau}
    [*(\Blue)]
    3 & *(\Red)5 & 7 & 9 & 10 & 11\\
    *(\Red)2 & *(\Red)4 & 6 & 8\\
    1
\end{ytableau}
\qquad
\begin{ytableau}
    [*(\Blue)]
    *(\Red)4 & 6 & 8 & 10 & 11\\
    2 & *(\Red)5 & 7 & 9\\
    *(\Red)1 & 3
\end{ytableau}
\]
   
\end{construction}

\begin{lemma}
    \label{lemma:alpha2=colambda1=3}

    Construction~\ref{constr:alpha2=colambda1=3} produces an $\alpha$-coloring of $Q^\uparrow$.
\end{lemma}

\begin{proof}\

    \begin{enumerate}
        \item In the case $\lambda'_2 = 1$, the condition $\lambda \neq (4,1,1)$ guarantees that $n \geq 7$.
        Therefore there is at least one \blue box to the right of the entry 6 (in the tail of $\Tup$).
        It follows that $\sigma \cdot \Tup$ is the standard hook tableau with column $(1,6,n)$, and with row $(1,\ldots,5)$ if $n = 7$, or row $(1, \ldots,5,7, \ldots, n-1)$ if $n \geq 8$.

        \item The given coloring produces a tableau $\sigma \cdot \Tup$ with first column $(1,4,x)$, where $x \geq 6$, and with second column $(2,5)$, and third column either $(3)$ or $(3,y)$, where $y \geq 6$.
        The rest of the tableau $\sigma \cdot \Tup$ follows Construction~\ref{constr:canonical cycle} and thus is standard.

        \item The given coloring produces a tableau $\sigma \cdot Q^\uparrow$ with first column $(1,3,5)$ and second column $(2,4,x)$ where $x \geq 6$.
        The rest of the tableau $\sigma \cdot \Tup$ follows Construction~\ref{constr:canonical cycle} and thus is standard. \qedhere
    \end{enumerate} 
\end{proof}

\begin{construction}
    \label{constr:2-column shapes}
    Here we assume that $\alpha_2 \geq 3$, and that $\lambda$ has exactly two columns.
    We exclude the case where $\alpha = (5,3)$ and $\lambda = (2^4)$, since it occurs in the table in Theorem~\ref{thm:main result}.
    (We also omit the case $\alpha_2 = \lambda'_1 = 3$, which is handled in Construction~\ref{constr:alpha2=colambda1=3}; this forces $\lambda'_1 \geq 4$ in the present construction.)

    \begin{enumerate}
        \item If $\lambda'_2 = 2$, meaning that $\lambda = (2^2,1^{n-4})$, we have two cases:
        
        \begin{enumerate}
            \item If $\alpha_1 = \alpha_2$ is odd, then there is no $\alpha$-coloring of $Q^\uparrow$ for any admissible $Q \in \SYT(\lambda)$.
            However, the shape is still obtainable (with a non-admissible $Q$).
        In this case, we exhibit the permutation
        \[
        \sigma = \left[n-1, n-2, \ldots, \frac{n}{2}+1, 1, n, \frac{n}{2}, \ldots, 3, 2 \right].
        \]
        As an example, for $\alpha = (5,5)$ and $\lambda = (2^2, 1^6)$, the permutation $\sigma$ has the following graph:
        
        \begin{center} 
\scalebox{.85}{
\begin{tikzpicture}[scale=.4, every node/.style={scale=.75},baseline=(current bounding box.center)]
        \draw[lightgray] (1,1) grid (10,10);
        \draw[dotted,thick] (1,1) -- (10,10);

        \draw [ultra thick, red!70!black] (1,9) node [dot,fill=black] {} -- (9,9) -- (9,3) node [dot,fill=black] {} -- (3,3) -- (3,7) node [dot,fill=black] {} -- (7,7) -- (7,5) node [dot,fill=black] {} -- (5,5) -- (5,1) node [dot,fill=black] {} -- (1,1) -- (1,9);
        
        \draw [ultra thick, cyan!20!blue] (2,8) node [dot,fill=black] {} -- (8,8) -- (8,4) node [dot,fill=black] {} -- (4,4) -- (4,6) node [dot,fill=black] {} -- (6,6) -- (6,10) node [dot,fill=black] {} -- (10,10) -- (10,2) node [dot,fill=black] {} -- (2,2) -- (2,8);

        \node at (1,0) {$1$};
        \node at (2,0) {$2$};
        \node at (3,0) {$3$};
        \node at (4,0) {$4$};
        \node at (5,0) {$5$};
        \node at (6,0) {$6$};
        \node at (7,0) {$7$};
        \node at (8,0) {$8$};
        \node at (9,0) {$9$};
        \node at (10,0) {$10$};
        
        \node at (0,1) {$1$};
        \node at (0,2) {$2$};
        \node at (0,3) {$3$};
        \node at (0,4) {$4$};
        \node at (0,5) {$5$};
        \node at (0,6) {$6$};
        \node at (0,7) {$7$};
        \node at (0,8) {$8$};
        \node at (0,9) {$9$};
        \node at (0,10) {$10$};
        
\end{tikzpicture}
}
\end{center}
        
        \item Otherwise, we take $Q = \Tdown$, and in $\Tup$ we color exactly one of the two boxes in column 2 \red.
        If $\alpha_2$ is even (resp., odd), then give column 1 the spiral coloring with \red (resp., \blue) on the outside, and color the top box \red (resp., \blue) in column 2.
        \end{enumerate}

        \item If $\lambda'_2 \geq 3$, then take $Q = \Tdown$, and color $\Tup$ as follows:

        \begin{enumerate}
        
        \item If $\lambda'_1 = \lambda'_2$, then we exclude the case $\alpha = (5,3)$, since it occurs in the table in Theorem~\ref{thm:main result}.
        There are exactly two \red boxes in column 1.
        Give column 2 the \emph{upside-down} spiral coloring with \blue on the outside.
        Give column 1 the spiral coloring where the outside color matches the color of the higher terminal box in column 2.

        \item If $\lambda'_1 > \lambda'_2$, then there is exactly one \red box in column 2, specifically one of the two top boxes.
        Give column~1 the spiral coloring with \red on the outside; the color of the lower terminal box in column 1 determines the color of the top box in column 2.
        \end{enumerate}
        
    \end{enumerate}
See the examples below.
\[
\alpha = (8,7) :
\qquad
\begin{ytableau}
    [*(\Blue)]
    13 & 15\\
    *(\Red)12 & *(\Red)14\\
    11\\
    *(\Red)10\\
    9\\
    *(\Red)8\\
    7\\
    *(\Red)6\\
    5\\
    *(\Red)4\\
    3\\
    *(\Red)2\\
    1
\end{ytableau}
\qquad
\begin{ytableau}
    [*(\Blue)]
    *(\Red)10 & *(\Red)15\\
    9 & 14\\
    *(\Red)8 & 13\\
    7 & 12\\
    *(\Red)6 & 11\\
    *(\Red)5\\
    4\\
    *(\Red)3\\
    2\\
    *(\Red)1\\
    \none\\
    \none\\
    \none
\end{ytableau}
\qquad
\begin{ytableau}
    [*(\Blue)]
    *(\Red)9 & 15\\
    8 & *(\Red)14\\
    *(\Red)7 & 13\\
    *(\Red)6 & 12\\
    *(\Red)5 & 11\\
    4 & 10\\
    *(\Red)3\\
    2\\
    *(\Red)1\\
    \none\\
    \none\\
    \none\\
    \none
\end{ytableau}
\qquad\qquad
\alpha = (8,6): \qquad 
\begin{ytableau}
    [*(\Blue)]
    *(\Red)12 & *(\Red)14\\
    11 & 13\\
    *(\Red)10\\
    9\\
    8\\
    7\\
    6\\
    *(\Red)5\\
    4\\
    *(\Red)3\\
    2\\
    *(\Red)1\\
    \none
\end{ytableau}
\qquad
\begin{ytableau}
    [*(\Blue)]
    7 & 14\\
    *(\Red)6 & *(\Red)13\\
    5 & 12\\
    4 & *(\Red)11\\
    3 & *(\Red)10\\
    *(\Red)2 & *(\Red)9\\
    1 & 8\\
    \none\\
    \none\\
    \none\\
    \none\\
    \none\\
    \none
\end{ytableau}
\]
    
\end{construction}

\begin{lemma}
    \label{lemma:2-column shapes}
    Construction~\ref{constr:2-column shapes} produces an $\alpha$-coloring of $\Tup$ (and, in case (1a), we have $\sigma \in \mathcal{C}_{\alpha}$ and $\sh(\sigma) = \lambda$).
\end{lemma}

\begin{proof}\

\begin{enumerate}
    \item \begin{enumerate}
        \item By applying the RS algorithm to the given one-line notation of $\sigma$, it is straightforward to verify that $\sh(\sigma) = (2,2,1,\ldots,1) = \lambda$.
        Translating the given one-line notation of $\sigma$ into cycle notation, we have
        \[
        \sigma = \left(1, n-1, 3, n-3, \ldots, \frac{n}{2}+2, \frac{n}{2} \right) \left(2, n-2, 4, n-4, \ldots, \frac{n}{2}-1, \frac{n}{2}+1 \right),
        \]
        so that $\sigma \in \mathcal{C}_{(n/2, \: n/2)} = \mathcal{C}_\alpha$.
        \item Suppose first that $\alpha_2$ is even.
        Then in column 1 of $\Tup$, the number of \red boxes is odd, and is less than or equal to the number of \blue boxes; thus with \red on the outside, the lower terminal box in column 1 is \red.
        Coloring the top box \red in column 2 therefore guarantees that the entries in the two terminal boxes of column 1 (in $\Tup$) appear in reverse order in column 2 (of $\sigma \cdot \Tup$).
        Likewise, the fact that \red is on the bottom of column 1 guarantees that the two entries in column 2 (of $\Tup$) appear in reverse order at the bottom of column 1 (in $\sigma \cdot \Tup$).

        Now suppose that $\alpha_2$ is odd.
        Then in column 1 of $\Tup$, the number of \red boxes is even, and is strictly less than the number of \blue boxes, since otherwise we would be in case (1a); thus with \blue on the outside, the lower terminal box in column 1 is \blue.
        (This would not have been true if $\alpha_1 = \alpha_2$.)
        The rest of the argument is the same as before, upon switching the roles of \blue and \red.
    \end{enumerate} 
    \item  

    \begin{enumerate}
    
    \item First suppose $\lambda'_1 = 4$.
    Since we exclude the case where $\alpha = (5,3)$ and $\lambda = (2^4)$, which appears in the table in Theorem~\ref{thm:main result}, this leaves only the case $\alpha = (4,4)$ with $\lambda = (2^4)$.
    It is straightforward to check that $\sigma \cdot \Tup$ has columns $(1,2,5,6)$ and $(3,4,5,7)$, and thus is standard.

    Now suppose $\lambda'_1 \geq 5$.
    Since $\lambda'_1 = \lambda'_2$, we have $\alpha_2 \leq \lambda'_2$, and hence there are at least two \blue boxes in column 2.
    The lower terminal box in column 1 is always \blue, regardless of the outer color.
    Therefore $\sigma$ sends this entry to the top box in column 2 (which is \blue), and sends the (greater) entry in the terminal \red box to the second box from the top of column 2 (which is red); thus the two top boxes in column 2 of $\sigma \cdot \Tup$ are in standard order.
    Since $\lambda$ has only two columns, the second column has no slashes (in the arrow scheme of Construction~\ref{constr:canonical cycle}), and thus the arrows spiral from the top box to the bottom box, from the next topmost box to the next bottommost box, etc., which is the vertical reflection of a spiral coloring (see Definition~\ref{def:spiral coloring} and Figure~\ref{fig:spiral coloring}).
    Therefore the upside-down spiral coloring reverses the order of the entries in column 2 as desired, and the \emph{higher} terminal box (which has the greater entry) is sent by $\sigma$ to the \emph{lower} of the two bottom boxes in column 1 (which we forced to match its color).
    Thus the bottom two boxes in column 1 have the standard order, and it follows that $\sigma \cdot \Tup$ is standard.

    \item We only need to check the bottom box in column 1 of $\sigma \cdot \Tup$, since the rest of the tableau is automatically standard by following the arrows in Definition~\ref{def: alpha coloring}.
    The bottom box in column 1 is \red, and thus has entry $n$ (if column 2 has \red on top) or $n-1$ (if column 2 has \blue on top, in which case $n$ remains in some \blue box of column 2).
    Thus this bottom entry is the greatest entry in column 1, and (since $\lambda'_1 > \lambda'_2$) is the only entry in its row.
    \qedhere

    \end{enumerate}
    
    \end{enumerate}
    
\end{proof}

\begin{construction}
\label{constr:hooks}
    Here we assume that $\alpha_2 \geq 3$ and that $\lambda$ is a hook (i.e., $\lambda_2 = 1$).
    We exclude the case where $\alpha = (n/2, \: n/2)$ and $\lambda = (n-2, \: 1, \: 1)$, since it occurs in the table in Theorem~\ref{thm:main result}.
    We also exclude the case where $\alpha = (n/2, \: n/2)$ with $n \mid 4$ and $\lambda = (3,1,\ldots,1)$, since it occurs in the table.
    (We also omit the case $\alpha_2 = \lambda'_1 = 3$, which is handled in Construction~\ref{constr:alpha2=colambda1=3}.)
    Take $Q = \Tdown$, and color $\Tup$ as follows:

    \begin{enumerate}
        \item If $\lambda'_1 = 3$, then the \red boxes have entries $2,4,6,\ldots,2\alpha_2$.
    
        \item If $\lambda_1 = 3$, then the coloring depends on the parity of $\alpha_2$.
        If $\alpha_2$ is odd, then color the entry $n$ \red, and for the remaining \red boxes, give column 1 the spiral coloring with \red on the outside.
        If $\alpha_2$ is even, then replace ``\red'' with ``\blue'' in the previous sentence.

        \item If $\lambda_1 \geq 4$ and $\lambda'_1 \geq 4$, then the number of \red boxes in column 1 is
        \[
        \min \Big\{ \lambda'_1 - 2, \;  \alpha_2 - 1 \Big\}.
        \]
        Give column 1 the spiral coloring with \blue on the outside.
        In the remaining tail, beginning with whichever color is the lower terminal box in column 1, alternate colors from left to right until one color is exhausted. 
    \end{enumerate}
See the examples below.
\[
\alpha = (8,7):
\qquad
\begin{ytableau}
    [*(\Blue)]
    *(\Red)15 & 14 & *(\Red)13\\
    12\\
    *(\Red)11\\
    10\\
    *(\Red)9\\
    8\\
    7\\
    6\\
    *(\Red)5\\
    4\\
    *(\Red)3\\
    2\\
    *(\Red)1
\end{ytableau}
\qquad
\begin{ytableau}
    [*(\Blue)]
    7 & *(\Red)8 & 9 & *(\Red)10 & 11 & 12 & 13 & 14 & 15\\
    *(\Red)6\\
    *(\Red)5\\
    *(\Red)4\\
    *(\Red)3\\
    *(\Red)2\\
    1\\
    \none\\
    \none\\
    \none\\
    \none\\
    \none\\
    \none
\end{ytableau}
\qquad
\begin{ytableau}
    [*(\Blue)]
    3 & *(\Red)4 & 5 & *(\Red)6 & 7 & *(\Red)8 & 9 & *(\Red)10 & 11 & *(\Red)12 & 13 &*(\Red)14 & 15\\
    *(\Red)2\\
    1\\
    \none\\
    \none\\
    \none\\
    \none\\
    \none\\
    \none\\
    \none\\
    \none\\
    \none\\
    \none
\end{ytableau}
\]

\end{construction}

\begin{lemma}
    \label{lemma:hook}
    Construction~\ref{constr:hooks} produces an $\alpha$-coloring of $\Tup$.
\end{lemma}

\begin{proof}\

\begin{enumerate}
    \item Since here we exclude the case $\alpha = (n/2, \: n/2)$, we have $\alpha_1 > \alpha_2$, which guarantees that the entry $n$ in $\Tup$ is colored \blue.
    The given coloring then permutes the entries of each color clockwise around the hook, so that $\sigma \cdot \Tup$ is the hook with column $(1,4,n)$ and row $(1,2,3,5,\ldots,n-1)$.
    \item Suppose $\alpha_2$ is odd.
    Then in column 1 of $\Tup$, the number of \red boxes is even, and is at most the number of \blue boxes; hence with \red on the outside of the spiral coloring, the lower terminal box is \blue.
    Thus the two entries in these terminal boxes appear in increasing order in the last two boxes of row 1 of $\sigma \cdot \Tup$.
    In turn, since \red is at the bottom of column 1, the \red entry $n$ appears at the bottom of column 1 (in $\sigma \cdot \Tup$), while the \blue entry $n-1$ appears just above this.
    Hence $\sigma \cdot \Tup$ is standard.

    Now suppose $\alpha_2$ is even; then we exclude the case $\alpha = (n/2, \: n/2)$ since this would imply that $4 \mid n$ with $\lambda = (3,1,\ldots,1)$, listed in the table in Theorem~\ref{thm:main result}.
    Thus we have $\alpha_1 > \alpha_2$.
    Therefore in column 1 of $\Tup$, the number of \red boxes is odd, and is strictly less than the number of \blue boxes; hence with \blue on the outside of the spiral coloring, the lower terminal box is \red.
    (This would have been false if $\alpha_1 = \alpha_2$.)
    The rest of the argument is identical to the one above.
    
    \item The spiral coloring with \blue on the outside is always possible because there are at least two \blue boxes in column~1, and the horizontal alternation is always possible because there is at least one \red box in the tail.
    It therefore suffices to prove that the entry $n$ in $\Tup$ is colored \blue, since then $n$ will appear at the bottom of column~1 in $\sigma \cdot \Tup$.
    
    If $\lambda'_1 - 2 \leq \alpha_2 - 1$, then the \blue boxes in column 1 are the top and bottom, and so the lower terminal box is \red, which is then the first color in the tail.
    Since in the tail the number of \blue boxes must be greater than or equal to the number of \red boxes, the entry $n$ must be colored \blue.
    On the other hand, if $\lambda'_1 > \alpha_2 - 1$, then there is exactly one \red box in the tail; since we assume $\lambda_1 \geq 4$, this tail has at least three boxes, of which the \red box is either first or second.
    Therefore the entry $n$ must be colored \blue. \qedhere
        
\end{enumerate}
    
\end{proof}

\begin{construction}
\label{constr:2}
    Here we assume that $\alpha_2 = 2$.
    We exclude $\lambda  = (2,1,1)$, $\lambda = (2,2,2)$, and $\lambda = (3,1)$, since these occur in the table in Theorem~\ref{thm:main result}.

    \begin{enumerate}
        \item If $\lambda'_1 \neq 3$, then let $Q = \Tdown$.
        The \red boxes in $\Tup$ have entries $\lceil \lambda'_1/2 \rceil$ and $\lceil \lambda'_1/2 \rceil + 1$.
        
        \item If $\lambda'_1 = 3$, then we have the following cases: 
        \begin{enumerate}
            \item If $\lambda'_2 = 1$, then let $Q = \Tdown$.
            The \red boxes in $\Tup$ have entries 2 and 4.
            \item If $\lambda'_2 \geq 2$, then let $Q$ be the tableau obtained from $\Tdown$ by transposing the entries 3 and 4.
            The \red boxes in $Q^\uparrow$ have entries 1 and 4.
    \end{enumerate}
    \end{enumerate}
See the examples below.
\[
\alpha = (5,2):
\qquad
\begin{ytableau}
    [*(\Blue)]
    4 & 6 & 7\\
    *(\Red)3 & 5\\
    *(\Red)2\\
    1
\end{ytableau}
\qquad
\begin{ytableau}
    [*(\Blue)]
    3 & *(\Red)4 & 5 & 6 & 7\\
    *(\Red)2\\
    1\\
    \none
\end{ytableau}
\qquad
\begin{ytableau}
    [*(\Blue)]
    *(\Red)4 & 5 & 7\\
    2 & 3 & 6\\
    *(\Red)1\\
    \none
\end{ytableau}
\qquad
\begin{ytableau}
    [*(\Blue)]
    *(\Red)4 & 6 & 7\\
    2 & 5\\
    *(\Red)1 & 3\\
    \none
\end{ytableau}
\]
    
\end{construction}

\begin{lemma}
    \label{lemma:2}
    Construction~\ref{constr:2} produces an $\alpha$-coloring of $Q^\uparrow$.
\end{lemma}

\begin{proof}\

    \begin{enumerate}
        \item 
        If $\lambda'_1 = 2$, then the given coloring produces the first column $(1,2)$ in $\sigma \cdot \Tup$, while the remaining entries (since Theorem~\ref{thm:box} forces $\lambda'_2 = 2$) are permuted according to Construction~\ref{constr:canonical cycle}, yielding a standard tableau.
        If $\lambda'_1 \geq 4$, then with the given coloring, $\sigma$ transposes the two \red entries in the middle of column 1, while reversing the surrounding entries (excluding the entry which it sends to column 2) in such a way that the entries above and below the \red boxes switch sides with each other.
        This yields a standard first column in $\sigma \cdot \Tup$ (since the bottom entry comes from some column to the right in $\Tup$).
        The rest of the tableau $\sigma \cdot \Tup$ is standard by Lemma~\ref{lemma:canonical cycle}.

        \item \begin{enumerate}
            \item This coloring transposes the entries $2$ and $4$.
            The assumption $\lambda \neq (2,1,1)$ guarantees that $n \geq 5$; thus in $\sigma \cdot \Tup$ the first column is $(1,4,n)$, while the first row is $(1,2,3,5,6,\ldots,n-1)$, yielding a standard (hook-shaped) tableau.

            \item If $m \geq 3$ then the second column of $Q^\uparrow$ is either $(5,3)$ or $(6,5,3)$ where the 3 is slashed; on the other hand, if $m=2$, then the assumption $\lambda \neq (2,2,2)$ forces the second column of $Q^\uparrow$ to be $(5,3)$ where the 3 is not slashed.
            In either case, $\sigma$ takes the entry 3 into the box with entry 2; therefore in $\sigma \cdot Q^\uparrow$, the first column is $(1,3,4)$.
            The second column has the entry 2 on top, and the rest of the tableau is permuted as in Construction~\ref{constr:canonical cycle} and is therefore standard. \qedhere
        \end{enumerate}
    \end{enumerate}
\end{proof}
    
\begin{construction}
\label{constr:1}
    Here we assume that $\alpha_2 = 1$.
    We exclude $\lambda = (n/2, \: n/2)$, since it occurs in the table in Theorem~\ref{thm:main result}.
    Let $Q = \Tdown$, and color $\Tup$ as follows:

    \begin{enumerate}
        \item If $\lambda'_1 \leq 2$, then the \red box has the entry $n$.
        \item If $\lambda'_1 \geq 3$, then the \red box has the entry $\lceil (\lambda'_1 + 1) / 2 \rceil$.
        \end{enumerate}
See the examples below.
\[
\alpha = (6,1):
\quad
\begin{ytableau}
    [*(\Blue)]
    2 & 4 & 5 & 6 & *(\Red)7\\
    1 & 3\\
    \none\\
    \none
\end{ytableau}
\qquad
\begin{ytableau}
    [*(\Blue)]
    4 & 6 & 7\\
    *(\Red)3 & 5\\
    2\\
    1
\end{ytableau}
\]

\end{construction}

\begin{lemma}
    \label{lemma:1}
    Construction~\ref{constr:1} produces an $\alpha$-coloring of $\Tup$.
\end{lemma}

\begin{proof}\

    \begin{enumerate}
        \item If $\lambda'_1 = 1$, then we must have $\alpha_1 = \alpha_2 = 1$ and $\lambda = (2)$, since otherwise $\mathcal{B}_\alpha$ contains no one-row shapes.
        Thus $\Tup = \ytableaushort{12}$ is already standard.
        The given coloring (i.e., the entry 1 is \blue and 2 is \red) means that $\sigma$ fixes both entries,  yielding a standard tableau.

        If $\lambda'_1 = 2$, then the coloring fixes the entry $n$ at the end of the first row, while the other entries are permuted according to Construction~\ref{constr:canonical cycle} and thus (by Lemma~\ref{lemma:canonical cycle}) form a standard subtableau.
        The entire tableau $\sigma \cdot \Tup$ is also standard, since the assumption $\lambda' \neq (n/2, \: n/2)$ guarantees that $n$ is the only entry in its column.

        \item Let $\lambda'_1 \geq 3$, and consider the arrows in Construction~\ref{constr:canonical cycle} restricted to the \blue boxes in column 1.
        The terminal box in this spiral of arrows is either the middle box (if the number of \blue boxes is odd, in which we place the \red box just below) or the box above the middle (if the number of \blue boxes is even, in which we place the \red box just above).
        In either case, by our placement of the \red box,  $\sigma$ fixes the \red entry and reverses the order of the \blue entries (excluding the entry in the terminal box): therefore column 1 is standard in $\sigma \cdot \Tup$.
        The rest of the tableau is permuted as in Construction~\ref{constr:canonical cycle}, and thus by Lemma~\ref{lemma:canonical cycle} the entire tableau $\sigma \cdot \Tup$ is standard. \qedhere
        \end{enumerate}
    \end{proof}

\section{Unattainable shapes}
\label{sec: unattainable}

Although the existence of an $\alpha$-coloring of $Q^\uparrow$ for some admissible $Q \in \SYT(\lambda)$ guarantees that $\lambda \in \mathcal{S}_\alpha$ (Lemma~\ref{lemma:assoc perm}), the converse is not true; see case (1) in Construction~\ref{constr:2-column shapes}.
Therefore, in order to prove that a given $\lambda \in \mathcal{B}_\alpha$ is not an element of $\mathcal{S}_\alpha$, we will argue from general principles that no $\sigma \in \mathcal{C}_\alpha$ can have $\lambda$ as its RS shape.
Note that the cases listed in the following lemma are precisely those given in the table in Theorem~\ref{thm:main result}.

\begin{lemma}
    \label{lemma:unattainable}
    For each of the following pairs $(\alpha,\lambda)$, we have $\lambda \notin \mathcal{S}_\alpha$:
    \begin{enumerate}
    \item $\alpha = (n-1, \: 1)$ and $\lambda = \left( \frac{n}{2}, \frac{n}{2} \right)$, where $2 \mid n$.
    \item $\alpha = \left( \frac{n}{2}, \frac{n}{2} \right)$ and $\lambda = (n-2, \: 1, 1)$, where $2 \mid n$.
    \item $\alpha = \left( \frac{n}{2}, \frac{n}{2} \right)$ and $\lambda = (3, 1, \ldots, 1)$, where $4 \mid n$.
    \item 
    \begin{enumerate} 
    \item $\alpha = (4,2)$ and $\lambda = (2,2,2)$.
    \item $\alpha = (5,3)$ and $\lambda = (2,2,2,2)$.
    \end{enumerate}
    \end{enumerate}
\end{lemma}

\begin{proof}\

    \begin{enumerate}
        \item Although this is a special case of Lemma~5.5 in~\cite{RS-complete-Rubinstein}, we provide a direct proof here for the sake of completeness.
        Set $m \coloneqq n/2$.
        By way of contradiction, suppose there exists $\sigma \in \mathcal{C}_{(n-1,1)}$ such that $\sh(\sigma) = \lambda = (m,m)$.
        The Young diagram of $\lambda$ is the following:
        \begin{center}
        \begin{tikzpicture}[scale=.3]
    
            \draw (0,0) grid (8,-2);
                    
            \draw [decoration={brace, mirror}, decorate] (-0.2,0) -- (-0.2,-2) node[midway, left=5pt] {$2$};
            \draw [decoration={brace},decorate] (0,.2) -- node[above=5pt] {$m$} (8,.2);    
        \end{tikzpicture}
        \end{center}
        
        By Greene's theorem \eqref{eq: desc asc description of cols and rows of RSK shape}, the one-line notation of $\sigma$ has a maximal ascending subsequence of length $m$, and maximal descending subsequence of length $2$.
        Let $u$ be the fixed point of $\sigma$ (i.e., $\sigma_u = u$), and write write $\sigma = \tau \upsilon$ as a product of commuting cycles, where $\tau$ has length $2m-1$ and $\upsilon = (u)$ is the fixed point. 
        Let $\sigma|_\tau = [\sigma_1, \ldots, \sigma_{u-1}, \sigma_{u+1}, \ldots, \sigma_n]$ and $\sigma|_\upsilon=[u]$ denote the one-line notation of $\sigma$ restricted to those elements in $\tau$ and $\upsilon$, respectively.

        Note that the first $u-1$ many elements of $\sigma|_\tau$ must all be less than $u$, and the remaining elements must all be greater than $u$ (otherwise $\sigma$ would contain a descending subsequence of length $3$).
        Therefore, the longest ascending subsequence of $\sigma$ (which we know has length $m$) is one element longer than the longest ascending subsequence of $\sigma|_\tau$.
        But this implies that the longest ascending subsequence of $\sigma|_\tau$ has length $m-1$, which is less than half the length of $\sigma|_\tau$.
        This is a contradiction, since Greene's theorem~\eqref{eq: desc asc description of cols and rows of RSK shape} guarantees that $\sigma|_\tau$ can be written as the disjoint union of two ascending subsequences.

        \item Set $m \coloneqq n/2$.
        By way of contradiction, suppose there exists $\sigma \in \mathcal{C}_{(m,m)}$ such that $\sh(\sigma) = \lambda = (n-2,1,1)$.
        The Young diagram of $\lambda$ is the following:
        \begin{center}
       \begin{tikzpicture}[scale=.3]
    
    \draw (0,0) grid (8,-1);
    \draw (0,0) grid (1,-3);

    \draw [decoration={brace, mirror}, decorate] (-0.2,0) -- (-0.2,-3) node[midway, left=5pt] {$3$};
    \draw [decoration={brace},decorate] (0,.2) -- node[above=5pt] {$2(m-1)$} (8,.2);    
\end{tikzpicture}
\end{center}
        
        By Greene's theorem \eqref{eq: desc asc description of cols and rows of RSK shape}, the one-line notation of $\sigma$ has a maximal ascending subsequence of length $2(m-1)$.
        Now write $\sigma = \tau \upsilon$ as a product of commuting cycles (using the notation from part (1) above). 
        By Theorem~\ref{thm:box}, both $\sigma|_\tau$ and $\sigma|_\upsilon$ have a maximal ascending subsequence of length at most $m-1$; but in fact this length is exactly $m-1$, since we established that $\sigma$ has an ascending subsequence of length $2(m-1)$.
        Thus the union of these two ascending subsequences of length $m-1$ (one from $\sigma|_\tau$, one from $\sigma|_\upsilon$) must form a maximal ascending subsequence in $\sigma$ of length $2(m-1)$, while the remaining two elements of $\sigma$ form a descending subsequence.

        Let $t_1 < \cdots < t_m$ and $u_1 < \cdots < u_m$ be the ordered lists of the elements in $\tau$ and $\upsilon$, respectively.
        Since the longest ascending subsequence of $\sigma|_\tau$ has length $m-1$, and since $\tau$ has no fixed points (being a cycle), it is straightforward to verify that either
        \begin{enumerate}
            \item[I.]  $\tau = (t_1, \ldots, t_m)$ and thus $\sigma|_\tau = [t_2, t_3, \ldots, t_m, t_1]$, or
            \item[II.] $\tau = (t_m, \ldots, t_1)$ and thus $\sigma|_\tau = [t_m, t_1, t_2, \ldots, t_{m-1}]$.
        \end{enumerate}
        The same is true of $\upsilon$, which leaves four (essentially identical) cases to examine; we give details here only for the case where both $\tau$ and $\upsilon$ are of type I.
    
    By the first paragraph, $\sigma$ has an ascending subsequence consisting of the (re-ordered) elements $t_2,\ldots,t_m, u_2,\ldots,u_m$, and a descending subsequence consisting of the (re-ordered) elements $t_1,u_1$.
    Without loss of generality, assume that $t_m < u_m$ (otherwise interchange the role of $\tau$ and $\upsilon$); then since $t_1$ and $u_1$ form a decreasing subsequence in $\sigma$, and since $\sigma_{t_m} = t_1$ and $\sigma_{u_m} = u_1$, we must have $t_1 > u_1$.
    Thus since $\sigma_{t_1} = t_2$ and $\sigma_{u_1} = u_2$, we must also have $t_2 > u_2$.
    This same argument cascades all the way up until it forces $t_m > u_m$, which is a contradiction.
    
    \item The argument is similar to that in part (2), upon switching the roles of ``ascending'' and ``descending.''
    Set $m \coloneqq n/2$, and set $\ell \coloneqq n/4 = m/2$ (since we assume $4 \mid n$).
    By way of contradiction, suppose there exists $\sigma \in \mathcal{C}_{(m,m)}$ such that $\sh(\sigma) = \lambda = (3,1,\ldots,1)$.   
    The Young diagram of $\lambda$ is the following:
        \begin{center}
       \begin{tikzpicture}[scale=.3]
    
    \draw (0,0) grid (3,-1);
    \draw (0,0) grid (1,-5);

    \draw [decoration={brace, mirror}, decorate] (-0.2,0) -- (-0.2,-5) node[midway, left=5pt] {$2(m-1)$};
    \draw [decoration={brace},decorate] (0,.2) -- node[above=5pt] {$3$} (3,.2);    
\end{tikzpicture}
\end{center}
    In the notation of part (2), Greene's theorem implies that both $\sigma|_\tau$ and $\sigma|_\upsilon$ have a maximal descending subsequence of length $m-1$, and the union of these subsequences is a maximal descending subsequence in $\sigma$ of length $2(m-1)$.
    Moreover, the remaining two elements form an ascending subsequence of $\sigma$.

    Since $\tau$ has a maximal descending subsequence of length $m-1$ and no fixed points, it is straightforward to verify that either
    \begin{enumerate}
    \item[I.] \quad $\begin{aligned}[t]
    \tau &= (t_1, t_m, t_2, t_{m-1}, t_3, t_{m-2}, \ldots, t_{\ell}, t_{\ell+1}) \text{  and thus} \\
    \sigma|_\tau &= [t_m, t_{m-1}, \ldots, t_{\ell+1}, t_1, t_\ell, t_{\ell-1}, \ldots, t_2], \text{  or}
    \end{aligned}$

    \item[II.] \quad $\begin{aligned}[t]
    \tau &= (t_1, t_m, t_\ell, t_{\ell+1}, t_{\ell-1}, t_{\ell+2}, t_{\ell-2}, t_{\ell+3}, \ldots, t_2, t_{m-1}) \text{  and thus} \\
    \sigma|_\tau &= [t_m, t_{m-1}, \ldots, t_{\ell+1}, t_{\ell-1}, t_{\ell-2}, \ldots, t_1, t_\ell],
    \end{aligned}$
    \end{enumerate}
    or else $\tau$ is the inverse of either type I or II.
    The same is true of $\upsilon$.
    Again, since all resulting cases are similar in essentials, we give details here for the case where both $\tau$ and $\upsilon$ are of type I.

    Since $\sigma_{t_{\ell+1}} = t_1$ and $\sigma_{u_{\ell+1}} = u_1$ are not part of the maximal descending subsequence in $\sigma|_\tau$ and $\sigma|_\upsilon$, respectively, the elements $\sigma_{t_{\ell+1}}$ and $\sigma_{u_{\ell+1}}$ form an ascending subsequence in $\sigma$, while the remaining elements form a maximal descending subsequence.
    Without loss of generality, assume that $t_{\ell+1} < u_{\ell+1}$, which forces $t_1 < u_1$.
    In turn, this forces the following cascade of inequalities (apparent from inspecting the cycle notation of $\tau$ and $\upsilon$ in type I):
    \[
    t_m > u_m \Longrightarrow t_2 < u_2 \Longrightarrow t_{m-1} > u_{m-1} \Longrightarrow \cdots \Longrightarrow t_\ell < u_\ell \Longrightarrow t_{\ell+1} > u_{\ell+1},
    \]
    which is a contradiction.
    
    \item Though a proof by contradiction is possible, its page length and lack of generalizing insight motivate us to simply prove this result by explicitly calculating $\mathcal{S}_{(4,2)}$ and $\mathcal{S}_{(5,3)}$ and observing that these shapes $\lambda$ do not appear. 
    The details are omitted. \qedhere
    \end{enumerate}
\end{proof}

\begin{proof}[Proof of Theorem~\ref{thm:main result}]

By Lemma~\ref{lemma:assoc perm}, an $\alpha$-coloring of a tableau $Q^\uparrow$ (for some admissible $Q \in \SYT(\lambda)$) guarantees that $\lambda \in \mathcal{S}_\alpha$.
The result now follows immediately from the lemmas in Section~\ref{sec:alpha-colorings} (verifying the constructions of $\alpha$-colorings for all pairs $(\alpha,\lambda)$ not listed in the table in Theorem~\ref{thm:main result}, with the exception of the special case $(\alpha, \lambda)$ verified directly in part (1a) of Lemma~\ref{lemma:2-column shapes}), combined with Lemma~\ref{lemma:unattainable} (verifying that the table is correct).
\end{proof}

\section{Final remarks and conjectures}
\label{sec:conjectures}

In this section, we record several conjectures and avenues for further research, regarding the general case where $r > 2$. 
We begin by treating cycle types corresponding to \emph{strict partitions} (i.e., partitions $\alpha$ with distinct parts $\alpha_1 > \cdots > \alpha_r$).
Somewhat surprisingly, these cycle types are far easier to handle than those corresponding to non-strict partitions.

\subsection*{Cycle types corresponding to strict partitions}

The following conjecture (verified by computer for all $n \leq 15$) is the natural generalization of Theorem~\ref{thm:main result} for strict partitions $\alpha$.

\begin{conjecture} \label{conj: distinct ct}
Let $\alpha = (\alpha_1, \ldots, \alpha_r) \vdash n$ be a strict partition,
where $r \geq 3$.
\begin{enumerate}
    \item If $\alpha_r > 1$, then $\mathcal{S}_\alpha=\mathcal{B}_\alpha$.
    \item If $\alpha_r = 1$ and $n$ is odd, then $\mathcal{S}_\alpha=\mathcal{B}_\alpha$.
     \item If $\alpha_r = 1$ and $n$ is even, then $\mathcal{B}_\alpha\setminus\mathcal{S}_\alpha=\{(\frac{n}{2},\frac{n}{2})\}$.
\end{enumerate}
\end{conjecture}

Next we make a general conjecture asserting the existence of $\alpha$-colorings if $\alpha$ is a strict partition.
Recall that for $r=2$, there is only one case in which $\lambda \in \mathcal{S}_\alpha$ does \emph{not} imply the existence of an admissible $Q \in \SYT(\lambda)$ such that $Q^\uparrow$ admits an $\alpha$-coloring; this was case~(1a) of Construction~\ref{constr:2-column shapes}, in which $\alpha = \left(\frac{n}{2}, \: \frac{n}{2}\right)$ was \emph{not} a strict partition.
The following conjecture (verified by computer for all $n \leq 13$) gives the generalization for arbitrary $r$.

\begin{conjecture}
    \label{conj:coloring}
    Let $\alpha = (\alpha_1, \ldots, \alpha_r) \vdash n$ be a strict partition, where $r \geq 1$.
    For each $\lambda \in \mathcal{S}_\alpha$, there exists an admissible $Q \in \SYT(\lambda)$ such that $Q^\uparrow$ admits an $\alpha$-coloring.
\end{conjecture}

\begin{figure}[t] 
     \centering
\includegraphics[width=.9\textwidth]{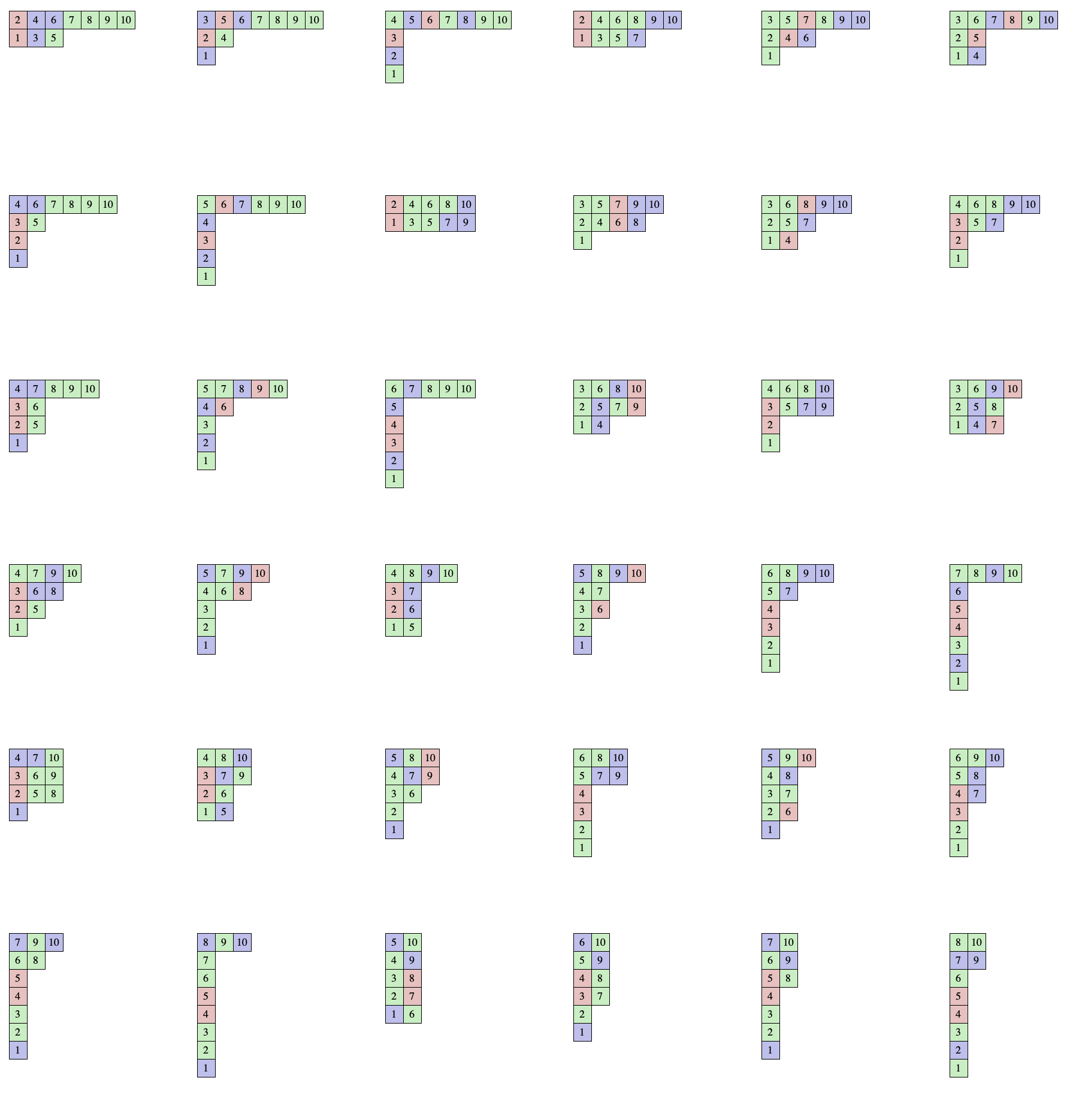}     
\caption{For  $\alpha=(5,3,2)$, we have $\mathcal{S}_\alpha=\mathcal{B}_\alpha$. Furthermore, for every $\lambda \in \mathcal{B}_\alpha$ there exists an $\alpha$-coloring of $\Tup$ as shown.
Here the colors $c_1, c_2, c_3$ are green, blue, and red, respectively.}
        \label{fig:alpha-colorings 532}
\end{figure}

In fact, it is almost always possible to choose $Q = \Tdown$ in Conjecture~\ref{conj:coloring}.
For example, when $n=13$ and $r\not=2$, there exists an $\alpha$-coloring of $\Tup$ for every $\lambda \in \mathcal{S}_\alpha = \mathcal{B}_\alpha$.
See Figure~\ref{fig:alpha-colorings 532}, where $\alpha = (5,3,2)$ and we show an $\alpha$-coloring of $\Tup$ for every $\lambda \in \mathcal{B}_\alpha$.

\subsection*{Cycle types corresponding to non-strict partitions}

It seems likely that the converse of Conjecture~\ref{conj: distinct ct} is also true.
In particular, this would mean that when $\alpha_r > 1$, we have  $\mathcal{S}_\alpha=\mathcal{B}_\alpha$ only if $\alpha$ is a strict partition.
By contrast, when $\alpha$ is not a strict partition (i.e., when $\alpha$ contains repeated parts), it becomes much more difficult to describe $\mathcal{S}_\alpha$.
For example, the case $\alpha = (3^r)$ is already quite complicated, even for small values of $r$.
However, using the fact that $(4,1^2) \notin \mathcal{S}_{(3^2)}$, it is possible to show that the only hook that can appear in $\mathcal{S}_{(3^r)}$ is of the form $(r+1, 1^{2r-1})$. The rest of the structure remains mysterious.

\begin{problem}
For general cycle types $\alpha = (\alpha_1, \ldots, \alpha_r)\vdash n$, describe $\mathcal{S}_\alpha$ explicitly.    
\end{problem}

There is at least one case that affords an explicit description of $\mathcal{S}_\alpha$, namely the case where $\alpha_1 \leq 2$.
 The following proposition is a consequence of Sch\"utzenberger~\cite{Schutzenberger}*{4.4 on p. 95} classifying the RS shapes that arise when $\sigma$ is an involution.
 
\begin{proposition}
\label{prop:schutz}
Let $\alpha = (2^{r-k}, \; 1^{k})$.
Then
\[
\mathcal{S}_{\alpha} = \left\{ \lambda \vdash n : \textup{$\lambda$ has exactly $k$ many columns of odd length} \right\}.
\]
\end{proposition}

\begin{figure}[t] 
     \centering
\includegraphics[width=.9\textwidth]{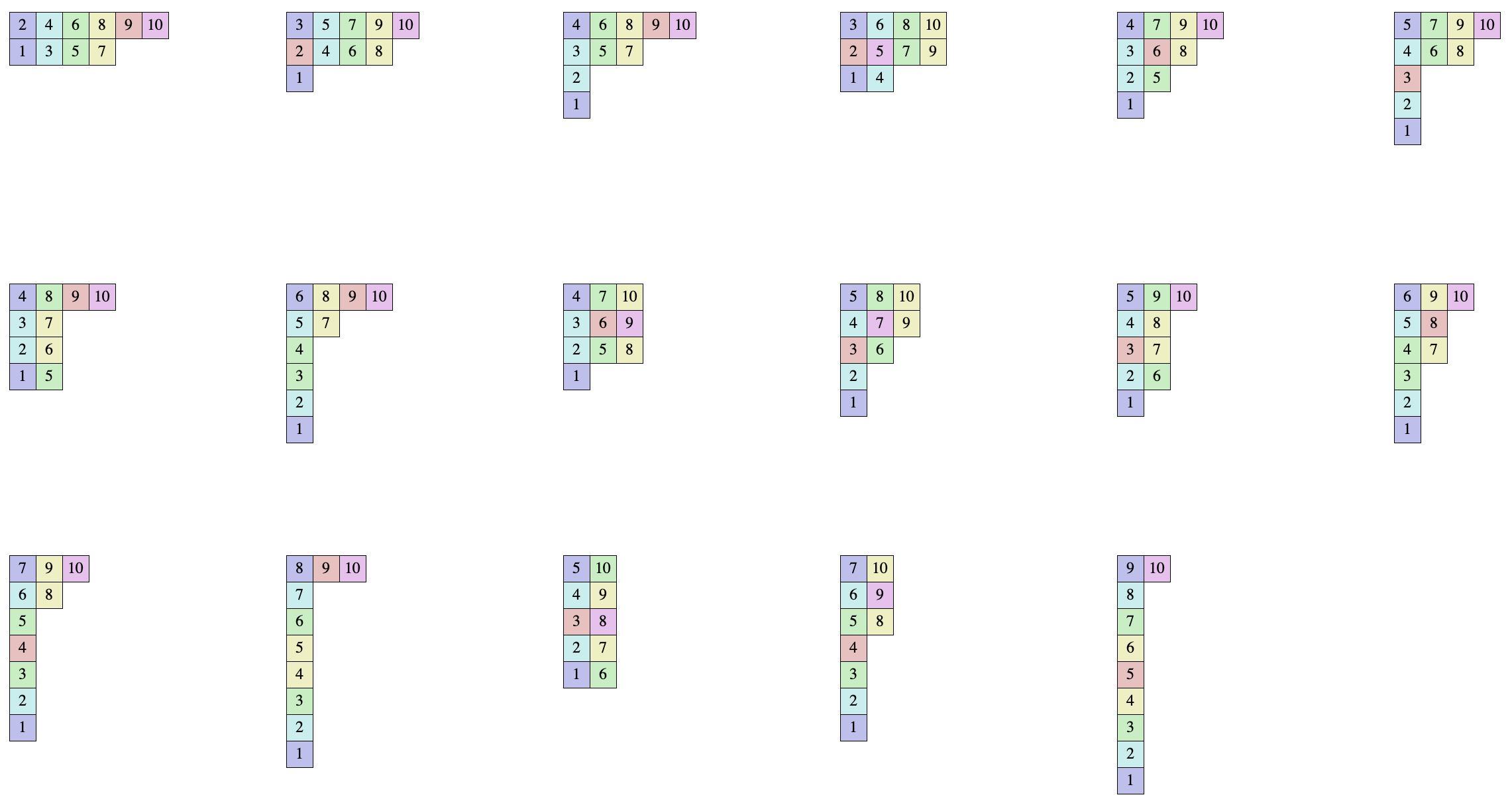}     
\caption{For $\alpha=(2^4,1^2)$, by Proposition~\ref{prop:schutz} we have the explicit description $\mathcal{S}_\alpha=\{\lambda\in \mathcal{B}_\alpha : \text{$\lambda$ has exactly $2$ columns of odd length}
\}$. 
For each $\lambda \in \mathcal{S}_\alpha$, we show the canonical $\alpha$-coloring of $\Tup$.
The red and the magenta boxes correspond to the two 1's in $\alpha$, while the remaining four colors occur in pairs that are placed symmetrically in each column.}
        \label{fig:schutz}
\end{figure}

Our $\alpha$-colorings provide an especially nice interpretation of Proposition~\ref{prop:schutz}.
In particular, if $\alpha = (2^{r-k}, 1^k)$, then for each $\lambda \in \mathcal{S}_\alpha$ there is a canonical $\alpha$-coloring of $\Tup$, obtained as follows.
In each column of odd length, color the middle box with one of the colors occurring in exactly one box (i.e., the $k$ many colors corresponding to the $1$'s in $\alpha$).
Then apply the remaining $r-k$ many colors (each occurring in a pair of boxes) such that the colors in each column are vertically symmetrical.
(See Figure~\ref{fig:schutz} for an example.)

Next, as a generalization of Proposition~\ref{prop:schutz}, we consider the effect of adding fixed points to permutations (i.e., appending $1$'s to cycle types).
In algebraic combinatorics, the \emph{Pieri rule} gives the expansion of the product of two Schur functions, one of which corresponds to a partition with a single row (say of size $k$):
\[
s_\mu \cdot s_{(k)} = \sum_{\lambda} s_\lambda,
\]
where the sum ranges over all partitions $\lambda$ whose the Young diagram is obtained from that of $\mu$ by adding $k$ many boxes, no two of which are added to the same column.
(The Pieri rule also gives an analogous result for single columns $(1^k)$.)
In the context of our problem, adding $k$ many fixed points to a cycle type $\alpha$ has a Pieri-like effect on the resulting RS shapes, described in the following conjectures (the second of which is stated in terms of $\alpha$-colorings).

\begin{conjecture}[``Almost'' Pieri rule] \label{conj: almost pieri}
Let $\alpha=(\alpha_1, \ldots, \alpha_r) \vdash n$,
where
$\alpha_r > 1$ 
and $r \geq 1$. 
Fix a positive integer $k$, and set $\widetilde{\alpha} \coloneqq (\alpha_1, \ldots, \alpha_r,1^k)$.
Then we have
\[
\mathcal{S}_{\widetilde{\alpha}} = \left\{ \lambda \in \mathcal{B}_{\widetilde{\alpha}} :
\begin{array}{l}
\textup{the Young diagram of $\lambda$ is obtained from the Young diagram} \\
\textup{of some $\mu \in \mathcal{S}_\alpha$ by adding exactly $k$ boxes, with no two} \\
\textup{in the same column, and such that if $\lambda_1' = 2$ then all $k$ boxes} \\
\textup{must be added to the first row}
\end{array} \right\}.
\]    
\end{conjecture}

\begin{conjecture}
\label{conj:alpha-coloring}
    Let $\alpha$ and $\widetilde{\alpha}$ be as in Conjecture~\ref{conj: almost pieri}.
    Let $\mu \vdash n$ such that there is an admissible $Q \in \SYT(\mu)$ where $Q^\uparrow$ admits an $\alpha$-coloring.
    If $\lambda$ is obtained from $\mu$ as described in Conjecture~\ref{conj: almost pieri}, then there exists an admissible $\widetilde{Q} \in \SYT(\lambda)$ such that $\widetilde{Q}^\uparrow$ admits an $\widetilde{\alpha}$-coloring.
\end{conjecture}

A sketch of a proof of Conjecture~\ref{conj:alpha-coloring} --- at least in the case where $Q = T_\mu$, which is almost always possible --- can be given as follows.
Consider an $\alpha$-coloring of $T_\mu^\uparrow$, with associated permutation $\sigma$.
To obtain an $\widetilde{\alpha}$-coloring of $\Tup$, we take the $\alpha$-coloring of $T_\mu^\uparrow$ and insert one box (a fixed point of the new associated permutation $\widetilde{\sigma}$) into each of the the $k$ many columns required to obtain the shape $\lambda$; in particular, we insert this box at the ``axis of symmetry'' in that column (as in case (2) of Construction~\ref{constr:1}), i.e., the unique location about which $\sigma$ reflects the entries in that column.
We then give each of these $k$ boxes a unique color.
Of course, special care must be taken in columns where $\sigma$ leaves no original entries remaining, but in general every column has a location at which a fixed point can be inserted so as to preserve the standardness of $\widetilde{\sigma} \cdot \Tup$.

This project also suggests several interesting enumerative problems:

\begin{problem}
    Given cycle type $\alpha \vdash n$, describe the \emph{multiset} of shapes $\{\sh(\sigma) : \sigma \in \mathcal{C}_\alpha \}$.
\end{problem}

\begin{problem}
\label{prob:admissible}
    Find a formula for the number of admissible tableaux of a given shape $\lambda$.
\end{problem}

\begin{problem}
    Given $\lambda \in \mathcal{S}_\alpha$, determine the number of $\alpha$-colorings of $\Tup$.
    More generally, given an admissible $Q \in \SYT(\lambda$), determine the number of $\alpha$-colorings of~$Q^\uparrow$.
\end{problem}

\bibliographystyle{amsplain}
\bibliography{refs}

\begin{bibdiv}
\begin{biblist}

\bib{Fulton-YT}{book}{
      author={Fulton, W.},
       title={Young tableaux},
      series={London Mathematical Society Student Texts},
   publisher={Cambridge University Press, Cambridge},
        date={1997},
      volume={35},
        ISBN={0-521-56144-2; 0-521-56724-6},
        note={With applications to representation theory and geometry},
      review={\MR{1464693}},
}

\bib{GesselReutenauer}{article}{
      author={Gessel, I.},
      author={Reutenauer, C.},
       title={Counting permutations with given cycle structure and descent set},
        date={1993},
        ISSN={0097-3165},
     journal={J. Combin. Theory Ser. A},
      volume={64},
      number={2},
       pages={189\ndash 215},
         url={https://www.sciencedirect.com/science/article/pii/009731659390095P},
}

\bib{RS-complete-Rubinstein}{article}{
      author={Goel, A.},
      author={Rubinstein-Salzedo, S.},
       title={R{S}-complete cycle types},
        date={2024},
        ISSN={1034-4942,2202-3518},
     journal={Australas. J. Combin.},
      volume={89},
       pages={215\ndash 233},
      review={\MR{4758567}},
}

\bib{Greene--Schensted-Extension}{article}{
      author={Greene, C.},
       title={An extension of {S}chensted's theorem},
        date={1974},
        ISSN={0001-8708},
     journal={Advances in Math.},
      volume={14},
       pages={254\ndash 265},
         url={https://doi.org/10.1016/0001-8708(74)90031-0},
      review={\MR{354395}},
}

\bib{Knuth-Perm-Mat-GYT}{article}{
      author={Knuth, D.~E.},
       title={Permutations, matrices, and generalized {Y}oung tableaux},
        date={1970},
        ISSN={0030-8730,1945-5844},
     journal={Pacific J. Math.},
      volume={34},
       pages={709\ndash 727},
         url={http://projecteuclid.org/euclid.pjm/1102971948},
      review={\MR{272654}},
}

\bib{Krattenthaler-Inc/dec-chains}{article}{
      author={Krattenthaler, C.},
       title={Growth diagrams, and increasing and decreasing chains in fillings of {F}errers shapes},
        date={2006},
        ISSN={0196-8858,1090-2074},
     journal={Adv. in Appl. Math.},
      volume={37},
      number={3},
       pages={404\ndash 431},
         url={https://doi.org/10.1016/j.aam.2005.12.006},
      review={\MR{2261181}},
}

\bib{Robinson-Representations-of-Sn}{article}{
      author={Robinson, G. de~B.},
       title={On the {R}epresentations of the {S}ymmetric {G}roup},
        date={1938},
        ISSN={0002-9327,1080-6377},
     journal={Amer. J. Math.},
      volume={60},
      number={3},
       pages={745\ndash 760},
         url={https://doi.org/10.2307/2371609},
      review={\MR{1507943}},
}

\bib{Sagan-Sn}{book}{
      author={Sagan, B.~E.},
       title={The symmetric group},
     edition={Second},
      series={Graduate Texts in Mathematics},
   publisher={Springer-Verlag, New York},
        date={2001},
      volume={203},
        ISBN={0-387-95067-2},
         url={https://doi.org/10.1007/978-1-4757-6804-6},
        note={Representations, combinatorial algorithms, and symmetric functions},
      review={\MR{1824028}},
}

\bib{SchenstedOriginalPaper}{article}{
      author={Schensted, C.},
       title={Longest increasing and decreasing subsequences},
        date={1961},
        ISSN={0008-414X,1496-4279},
     journal={Canadian J. Math.},
      volume={13},
       pages={179\ndash 191},
         url={https://doi.org/10.4153/CJM-1961-015-3},
      review={\MR{121305}},
}

\bib{Schutzenberger}{incollection}{
      author={Sch\"utzenberger, M.-P.},
       title={La correspondance de {R}obinson},
        date={1977},
   booktitle={Combinatoire et repr\'esentation du groupe sym\'etrique ({A}ctes {T}able {R}onde {CNRS}, {U}niv. {L}ouis-{P}asteur {S}trasbourg, {S}trasbourg)},
      series={Lecture Notes in Math.},
      volume={Vol. 579},
   publisher={Springer, Berlin-New York},
       pages={59\ndash 113},
      review={\MR{498826}},
}

\bib{StanleyEC2}{book}{
      author={Stanley, R.~P.},
       title={Enumerative combinatorics. {V}ol. 2},
     edition={2},
      series={Cambridge Studies in Advanced Mathematics},
   publisher={Cambridge University Press, Cambridge},
        date={2024},
      volume={208},
        ISBN={978-1-009-26249-1; 978-1-009-26248-4},
        note={With an appendix by Sergey Fomin},
      review={\MR{4621625}},
}

\bib{Viennot}{incollection}{
      author={Viennot, G.},
       title={Une forme g\'{e}om\'{e}trique de la correspondance de {R}obinson--{S}chensted},
        date={1977},
   booktitle={Combinatoire et repr\'{e}sentation du groupe sym\'{e}trique ({A}ctes {T}able {R}onde {CNRS}, {U}niv. {L}ouis-{P}asteur {S}trasbourg, {S}trasbourg, 1976)},
      series={Lecture Notes in Math., Vol. 579},
   publisher={Springer, Berlin},
       pages={29\ndash 58},
      review={\MR{0470059}},
}

\end{biblist}
\end{bibdiv}

\end{document}